\documentclass[12pt, reqno]{amsart}
\usepackage{amsmath,amsopn,amssymb,amsthm}
\usepackage{hyperref}
\usepackage[cp1251]{inputenc}
\usepackage[english]{babel}

\renewenvironment{proof}[1][Proof]{\textbf{#1.} }{\ \rule{0.5em}{0.5em}}

\DeclareMathOperator{\Id}{Id}

\DeclareMathOperator{\can}{can}

\DeclareMathOperator{\diag}{diag}

\DeclareMathOperator{\Ad}{Ad}

\DeclareMathOperator{\trace}{trace}
\DeclareMathOperator{\rk}{rk}

\DeclareMathOperator{\ccan}{can}
\DeclareMathOperator{\incl}{incl}

\renewenvironment{proof}[1][Proof]{\textbf{#1.} }
{\ \rule{0.5em}{0.5em}}

\newtheorem{theorem}{Theorem}
\newtheorem{prop}{Proposition}
\newtheorem{lemma}{Lemma}
\newtheorem{corollary}{Corollary}

\theoremstyle{definition}
\newtheorem{definition}{Definition}
\newtheorem{remark}{Remark}

\begin{document}

\title
[Generalized normal homogeneous Riemannian metrics \dots]
{Generalized normal homogeneous Riemannian metrics on spheres and projective spaces}
\author{V.N.~Berestovski\u\i, Yu.G.~Nikonorov}

\address{Berestovski\u\i\  Valeri\u\i\  Nikolaevich \newline
Omsk Branch of Sobolev Institute \newline
of Mathematics of SD RAS, \newline
Omsk, Pevtsov str., 13, \newline
644099, RUSSIA}
\email{berestov@ofim.oscsbras.ru}

\address{Nikonorov\ Yuri\u\i\  Gennadievich\newline
South Mathematical Institute of VSC RAS \newline
Vladikavkaz, Markus str., 22, \newline
362027, RUSSIA}
\email{nikonorov2006@mail.ru}

\thanks{The first author was supported in part by RFBR (grant 11-01-00081-a).
The second author was supported in part by the
State Maintenance Program for the Leading Scientific Schools of
the Russian Federation (grant NSh-921.2012.1) and by Federal Target
Grant ``Scientific and educational personnel of innovative
Russia'' for 2009-2013 (agreement no. 8206, application no. 2012-1.1-12-000-1003-014).}

\begin{abstract}
In this paper we develop new methods of study of generalized
normal homogeneous Riemannian manifolds.
In particular, we obtain a complete classification
of generalized normal homogeneous Riemannian metrics on spheres.
We prove that for any connected (almost effective) transitive on $S^n$
compact Lie group $G$, the family of $G$-invariant
Riemannian metrics on $S^n$ contains generalized normal homogeneous
but not normal homogeneous metrics if and only if this family depends
on more than one parameters. Any such family (that exists only for $n=2k+1$)
contains a metric $g_{\can}$ of constant sectional curvature $1$ on $S^n$. We also prove that
$(S^{2k+1}, g_{\can})$ is Clifford-Wolf homogeneous, and therefore generalized normal
homogeneous, with respect to $G$ (excepting the groups $G=SU(k+1)$ with odd $k+1$).
The space of unit Killing vector fields on $(S^{2k+1}, g_{\can})$ from Lie
algebra $\mathfrak{g}$ of Lie group $G$ is described as some symmetric space
(excepting the case $G=U(k+1)$ when one obtains the union of all
complex Grassmannians in $\mathbb{C}^{k+1}$).

\vspace{2mm}

\vspace{2mm} \noindent Key words and phrases: Clifford algebras, Clifford-Wolf homogeneous spaces, generalized normal
homogeneous Riemannian manifolds, g.o. spaces, Grassmannian algebra, homogeneous spaces, Hopf fibrations, normal homogeneous Riemannian manifolds, Riemannian submersions, symmetric spaces.
\end{abstract}

\maketitle

\section*{Introduction}

In papers \cite{BerNik, BerNik3, BNN} the authors studied a class of
generalized normal homogeneous Riemannian manifolds ($\delta$-homogeneous Riemannian manifolds, in other terms).

A metric space $(M,\rho)$ is \textit{$\delta$-homogeneous} (respectively, \textit{Clifford-Wolf homogeneous})
\cite{BP,BerNik} if for any points $x, y\in M$ there exists an isometry $f,$ \textit{$\delta$-$x$-translation}
(respectively, \textit{Clifford-Wolf translation}), of the space $(M,\rho)$ onto itself such that $f(x)=y$ and $f$
has maximal displacement at the point $x$ (respectively, equal displacements at all points), i.~e. for every point
$z\in M,$
$\rho(z,f(z))\leq \rho(x,f(x))=\rho(x,y)$ (respectively $\rho(z,f(z))=\rho(x,f(x))$). It is clear that any Clifford-Wolf
homogeneous metric space is $\delta$-homogeneous and any $\delta$-homogeneous metric space is homogeneous.
A connected Riemannian manifold $(M,\mu)$ is $\delta$-homogeneous (respectively, Clifford-Wolf homogeneous) if it
is $\delta$-homogeneous (respectively, Clifford-Wolf homogeneous) relative to its inner metric $\rho_{\mu}$.
In addition, it is \textit{$G$-$\delta$-homogeneous} or \textit{generalized $G$-normal homogeneous}
\cite{BNN} (respectively, \textit{$G$-Clifford-Wolf homogeneous}) if one can take isometries $f$ in definition
of $\delta$-homogeneity (respectively, Clifford-Wolf homogeneity) from Lie (sub)group $G$ of isometries of the
space $(M,\mu)$.

Any $\delta$-homogeneous Riemannian manifold is a \textit{geodesic orbit} and has nonnegative sectional curvature \cite{BerNik}. It follows from Theorem \ref{body} that any ($G$-)normal homogeneous Riemannian manifold \cite{Berg} is also ($G$-)$\delta$-homogeneous \cite{BerNik}. The converse statement is not true for some spaces \cite{BerNik, BNN}.
The existence of such spaces is always connected with unusual geometric properties of adjoint representations of corresponding Lie groups $G$ \cite{BerNik, BNN}. The complete classification of such simply connected Riemannian manifolds with positive Euler characteristic, indecomposable into direct metric product, is given in paper
\cite{BNN}. These are exactly all complex projective spaces with invariant Riemannian metrics
$(\mathbb{C}P^{2n+1}=Sp(n+1)/(U(1)\times Sp(n)),\nu_t),$ for $n\geq 1,$ $\frac{1}{2}< t<1,$ considered also in this paper, as well as the spaces homothetic to them. All previous results on generalized normal homogeneous Riemannian
manifolds have been collected, and revised in some respect, in our joint book \cite{BerNikBook}.

It was unknown up to now whether there are similar to $(\mathbb{C}P^{2n+1},\nu_t)$ connected compact homogeneous Riemannian manifolds with zero Euler characteristic (by the Hopf --- Samelson theorem \cite{HS}, every compact homogeneous manifold $G/H$ has nonnegative Euler characteristic).

In this paper we obtain known and new compact simply connected indecomposable generalized normal homogeneous, but not normal homogeneous, Riemannian manifolds. For this we use known results about normal invariant Riemannian metrics on compact homogeneous spaces and apply new construction of generalized normal homogeneous Riemannian metric from Theorem \ref{fireysum} and our results about Clifford-Wolf homogeneity of (round) Euclidean spheres with respect to different Lie groups from Theorems \ref{osn_sp} and \ref{CWhom_spin}. In particular, any sphere $S^{2n+1},$ $n\geq 2,$
(of zero Euler characteristic!) admits such metrics. In last sections we apply extensively Clifford algebras and some
reduced form of Grassmannian algebra. A part of these results was proved recently in texts \cite{B1}, \cite{B2} by quite different methods.

One of the most important results of this paper is a classification of all generalized normal homogeneous
metrics on spheres (see Theorems \ref{u_main}, \ref{su_main_del},  \ref{sp_sp_main}, \ref{sp_u_main},  \ref{sp_main},
\ref{new_spin}, Proposition \ref{su(2)n} and Conclusion).  We collect all these results in Table~2. To compare this classification with the classification of (usual) normal homogeneous metrics on spheres, we reproduce here Table~1
with the latter classification \cite[Table 2.4]{GZ}. We denote by $g_{\can}$ a Riemannian metric of constant
curvature $1$ on a given sphere. The space $\mathfrak{p}_3$ appears only in the last line of Tables 1 and 2; it is
$\mathfrak{u}(1)\subset \mathfrak{sp}(n+1)\oplus \mathfrak{u}(1),$ which is naturally isomorphic to the space, tangent
to the fiber of the Hopf fibration $pr_1$ (see Section \ref{invest}) at the point $(1,0,\dots,0)\in S^{4n+3}$.

The notation of parameters in Table 2 is adapted to notation in Table 1. In fact we use in this paper another notation,
namely, $c^2:=1,$ $t:=b^2,$ $s:=a^2,$ which coincides with the notation in papers \cite{Volp}, \cite{Volp1}, \cite{VZ}.
Then, summing up and comparing Tables 1 and 2, applying the notation from Section \ref{invest}, and adding at the end
the mentioned result from paper \cite{BNN}, we obtain the following statement.

\begin{theorem}
\label{nonnormal}
Up to homothety, the following spaces are all known at present generalized normal homogeneous,
but not normal homogeneous with respect to any transitive connected Lie group, Riemannian manifolds (for $n\geq 1$):
$$
(S^{2n+3},\xi_t), \, \frac{n+2}{2(n+1)}< t < 1; \quad  (S^{4n+3},\mu_t), \, \frac{1}{2}<t<1;
$$
$$
(S^{4n+3},\mu_{t,s}), \,\frac{1}{2}<t<1, \, 0<s<t; \quad (S^{15},\psi_t), \, \frac{1}{4}<t<1;
$$
$$
(\mathbb{C}P^{2n+1},\nu_t), \, \frac{1}{2}<t<1.
$$
In addition, all spheres above have zero Euler characteristic and the sphere $(S^5,\xi_t)$ has minimal dimension.
\end{theorem}

\renewcommand{\arraystretch}{1.5}

\begin{table}[h]\label{table1}
{\bf Table 1. Normal homogeneous metrics on spheres}
\begin{center}
\begin{tabular}
{|p{0.03\linewidth}|p{0.18\linewidth}|p{0.13\linewidth}|p{0.2\linewidth}|p{0.24\linewidth}|p{0.07\linewidth}|}
\hline
&$G$&$H$&$(\cdot,\cdot)|_{\mathfrak{p}_3}$&$(\cdot,\cdot)|_{\mathfrak{p}_2}$&$(\cdot,\cdot)|_{\mathfrak{p}_1}$\\
\hline \hline
1 & $SO(n+1)$&$SO(n)$&\multicolumn{3}{|c|}{$c^2 g_{\can}$}\\
\hline
2 & $G_2$ & $SU(3)$ & \multicolumn{3}{|c|}{$c^2 g_{\can}$}\\
\hline
3 &$Spin(7)$& $G_2$ & \multicolumn{3}{|c|}{$c^2 g_{\can}$} \\
\hline
4 & $SU(2)$ & $\{e\}$ & \multicolumn{3}{|c|}{$c^2 g_{\can}$} \\
\hline\hline
5 & $SU(n+1)$& $SU(n)$ & \multicolumn{2}{|c|}{$\frac{n+1}{2n} \,c^2 g_{\can}$} & $c^2 g_{\can}$ \\
\hline
6 & $U(n+1)$& $U(n)$ & \multicolumn{2}{|c|}{$b^2 c^2 g_{\can},\,\, b^2<\frac{n+1}{2n}$} & $c^2 g_{\can}$ \\
\hline
7 & $Sp(n+1)Sp(1)$& $Sp(n)Sp(1)$ &\multicolumn{2}{|c|}{$b^2 c^2 g_{\can},\,\, b^2<\frac{1}{2}$} & $c^2 g_{\can}$  \\
\hline
8 & $Spin(9)$ & $Spin(7)$ & \multicolumn{2}{|c|}{$\frac{1}{4} c^2 g_{\can}$} & $c^2 g_{\can}$  \\
\hline
9 & $Sp(n+1)$ & $Sp(n)$ & \multicolumn{2}{|c|}{$\frac{1}{2} c^2 g_{\can}$} & $c^2 g_{\can}$  \\
\hline\hline
10& $Sp(n+1)S^1$ & $Sp(n) S^1$ & $a^2 c^2 g_{\can},\,\, a^2<\frac{1}{2} $ & $\frac{1}{2} c^2 g_{\can}$ & $c^2 g_{\can}$ \\
\hline
\end{tabular}
\end{center}
\end{table}

\begin{table}[h]\label{table2}
{\bf Table 2. Generalized normal homogeneous metrics on spheres}
\begin{center}
\begin{tabular}
{|p{0.03\linewidth}|p{0.18\linewidth}|p{0.13\linewidth}|p{0.2\linewidth}|p{0.24\linewidth}|p{0.07\linewidth}|}
\hline
&$G$&$H$&$(\cdot,\cdot)|_{\mathfrak{p}_3}$&$(\cdot,\cdot)|_{\mathfrak{p}_2}$&$(\cdot,\cdot)|_{\mathfrak{p}_1}$\\
\hline \hline
1 & $SO(n+1)$&$SO(n)$&\multicolumn{3}{|c|}{$c^2 g_{\can}$}\\
\hline
2 & $G_2$ & $SU(3)$ & \multicolumn{3}{|c|}{$c^2 g_{\can}$}\\
\hline
3 &$Spin(7)$& $G_2$ & \multicolumn{3}{|c|}{$c^2 g_{\can}$} \\
\hline
4 & $SU(2)$ & $\{e\}$ & \multicolumn{3}{|c|}{$c^2 g_{\can}$} \\
\hline\hline
5 & $SU(n+1)$& $SU(n)$ & \multicolumn{2}{|c|}{$ b^2\,c^2 g_{\can},\,\, \frac{n+1}{2n}\leq b^2\leq 1$}& $c^2 g_{\can}$ \\
\hline
6 & $U(n+1)$& $U(n)$ & \multicolumn{2}{|c|}{$b^2 c^2 g_{\can},\,\, b^2\leq 1$} & $c^2 g_{\can}$ \\
\hline
7 & $Sp(n+1)Sp(1)$& $Sp(n)Sp(1)$ &\multicolumn{2}{|c|}{$b^2 c^2 g_{\can},\,\, b^2\leq 1$} & $c^2 g_{\can}$  \\
\hline
8 & $Spin(9)$ & $Spin(7)$ & \multicolumn{2}{|c|}{$b^2 c^2 g_{\can},\,\,\frac{1}{4}\leq b^2 \leq 1$} & $c^2 g_{\can}$  \\
\hline
9 & $Sp(n+1)$ & $Sp(n)$ & \multicolumn{2}{|c|}{$b^2 c^2 g_{\can},\,\,\frac{1}{2}\leq b^2 \leq 1$} & $c^2 g_{\can}$  \\
\hline\hline
10& $Sp(n+1)S^1$ & $Sp(n)S^1$ & $a^2 c^2 g_{\can},\, a^2\leq b^2$ & $b^2 c^2 g_{\can},\,\frac{1}{2}\leq b^2 \leq 1$ & $c^2 g_{\can}$ \\
\hline
\end{tabular}
\end{center}
\end{table}

The structure of this paper is as follows.
In the first section we describe homogeneous spaces and invariant metrics on these spaces
that are interesting for us in this study.

In the second section we
use a suitable characterization of generalized normal homogeneous Riemannian metrics
and the notion of (dual) $2$-means in the sense of W.J.~Firey in order to establish a new (quite unexpected)
method for generating metrics of the above type.

The third section is devoted to a study of Killing vector fields of constant length on round (Euclidean) spheres.
In particular, we prove that $S^{2k+1}$ is Clifford-Wolf homogeneous with respect to $U(k+1)$
and $S^{4k+3}$ is Clifford-Wolf homogeneous with respect to $Sp(k+1)$.

In Section 4 we obtain some auxiliary results on $\delta$-vectors and on the generalized homogeneity
of some Riemannian spaces.

In Section 5 we obtain a lot of new examples of generalized normal homogeneous Riemannian spaces
and prove some classification theorems.

In Section 6 we describe the spaces of unit Killing fields on round spheres from the Lie algebras
$\mathfrak{so}(2n)$, $\mathfrak{u}(n+1)$, $\mathfrak{su}(2(n+1))$ and $\mathfrak{sp}(n+1)$ as some
symmetric spaces (union of symmetric spaces in the second case).

Section 7 is devoted to a study of invariant metrics on the homogeneous space $Spin(9)/Spin(7)=S^{15}$.
In particular, we prove that the round sphere $S^{15}$ is Clifford-Wolf homogeneous with respect to
$Spin(9)$. (Similarly, one can prove that the normal homogeneous sphere $Spin(7)/G_2=S^7$ is Clifford-Wolf homogeneous with respect to $Spin(7)$, see Remark \ref{spin7}.) Besides, we classify all $Spin(9)$-generalized normal homogeneous Riemannian metrics on $S^{15}$.

In Section 8 we describe the structure of the spaces of Killing vector fields of constant length on the round sphere
$S^{15}$ from the Lie algebra $\mathfrak{spin}(9)$ and its Lie subalgebra $\mathfrak{spin}(8)$ as real Grassmannians
and discuss some related questions.

In conclusion section we add some remarks that (in particular) complete the classification
of generalized normal homogeneous Riemannian metric on spheres and projective spaces.

\smallskip

We thank Professor Wolfgang Ziller for useful discussions and Natalia Berestovskaya for help in preparation of the text.

\section{Homogeneous Riemannian manifolds being investigated}
\label{invest}

In this section we shall give a unified detailed description, in terms of some Hopf fibrations and parameters,
of all homogeneous Riemannian manifolds being investigated (see also \cite{Zil82, GlZ}).
Corresponding invariant Riemannian metrics are often called ``diagonal'' \cite{Volp, GZ}.

Let us consider at first the following classical Hopf fibrations
$$
\begin{array}{rlr}
p: & S^{2n+1}\rightarrow \mathbb{C}P^n, &n\geq 2,\\
pr: &S^{4n+3}\rightarrow \mathbb{H}P^n, &n\geq 1,\\
pr_1: &S^{4n+3}\rightarrow \mathbb{C}P^{2n+1}, &n\geq 1,\\
pro: &S^{15}\rightarrow \mathbb{C}aP^1=S^8(1/2),&\\
pr_2: &\mathbb{C}P^{2n+1}\rightarrow \mathbb{H}P^n, & n\geq 1.\\
\end{array}
$$
All these fibrations are Riemannian submersions with respect to canonical homogene\-ous Riemannian metrics $g_{\ccan}$
on the fibers, the base, and the total space. In the first four cases all total spaces are spheres with constant
sectional curvature $1$, fibers are totally geodesic spheres of dimension $1$, $3$, $1$, and $7,$ respectively, while
the bases in the first three cases (in the second case for $n\geq 2$) are supplied with the Fubini-Study metrics
(with sectional curvatures in the interval $[1/4,1]$), while the base in fourth case has constant sectional curvature $1/4$. In the last case both base and total space have the mentioned Fubini-Study metrics, and fibers are totally geodesic 2-spheres
(in other words, $\mathbb{C}P^1$).

All metrics on total spaces being investigated, except the third case, are
such variations of the canonical metric that metric tensor is multiplied by parame\-ter $t>0$ on fibers, while
all fibrations are still Riemannian submersions. As a result we obtain the following homogeneous
Riemannian manifolds: $(S^{2n+1}=U(n+1)/U(n),\xi_t)$, $(S^{4n+3}=Sp(n+1)\times Sp(1)/(Sp(n)\times Sp(1)),\mu_t)$,
$(S^{15}=Spin(9)/Spin(7),\psi_t)$, $(\mathbb{C}P^{2n+1}=Sp(n+1)/(U(1)\times Sp(n)),\nu_t)$. Now, the metric
$\mu_{t,s}$, $s>0$, on $S^{4n+3}$ is defined such that
\begin{equation}
\label{pr1}
pr_1: (S^{4n+3},\mu_{t,s})\rightarrow (\mathbb{C}P^{2n+1},\nu_t)
\end{equation}
is a Riemannian submersion while the metric tensor is multiplied by $s$ on circles-fibers,
so all these circles have length $2\pi\sqrt{s}$. Thus we get homogeneous Riemannian manifolds
$(S^{4n+3}=Sp(n+1)\times U(1)/(Sp(n)\times U(1)),\mu_{t,s})$. Note that $\mu_{1,s}=\xi_s$.
This is connected with natural inclusions
\begin{equation}
\label{incl}
Sp(n+1)\subset SU(2(n+1)), \quad Sp(n+1)\times U(1)\subset U(2(n+1)).
\end{equation}

Note that all invariant Riemannian metrics on corresponding homogeneous spaces above are proportional to
mentioned metrics with one exception: left-invariant metrics on the group $SU(2)=S^3$ (see Proposition \ref{su(2)n}).
Moreover, all Riemannian manifolds $(S^{2n+1},\xi_t)$, $(S^{4n+3},\mu_t)$, $(S^{15},\psi_t)$
are homothetic to \textit{distance spheres} in $(\mathbb{C}P^{n+1}, g_{\ccan})$, $(\mathbb{H}P^{n+1}, g_{\ccan})$, and
Caley plane $(\mathbb{C}aP^{2}, g_{\ccan}),$ respectively; the converse statement is also true \cite{Zil82}.
Normal homogeneous spaces $(S^{2n+1},\xi_t)$, where $0<t\leq (n+1)/2n$, $n\geq 1$, 
are also known as \textit{Berger's spheres} \cite{Berg, Zil77, Zil82}.
Finally, note that both $Sp(4)$ and $Spin(9)$ have dimension $36$,
which is minimal among dimensions of all Lie groups, acting transitively on $S^{15}$.

\section{Characterization of generalized normal homogeneous spaces}\label{gensec}

Let $M= G/H$ be a homogeneous space of a compact connected Lie group $G$.
Let us denote by $\langle \cdot ,\cdot \rangle$ a fixed  $\Ad(G)$-invariant Euclidean metric on the Lie algebra
$\mathfrak{g}$ of $G$ (for example, the minus Killing form if $G$ is semisimple) and by
\begin{equation}\label{reductivedecomposition}
\mathfrak{g}=\mathfrak{h} \oplus \mathfrak{p}
\end{equation}
the  associated $\langle \cdot ,\cdot \rangle$-orthogonal  reductive decomposition, where
$\mathfrak{h} = {\rm Lie}(H) $. An invariant Riemannian metric $g$
on $M$  is determined by an $\Ad(H)$-invariant inner product
$g_o =(\cdot,\cdot)$ on  the space $\mathfrak{p}$ which is identified with
the tangent space $M_o$ at the initial point $o = eH$.
If $\mathfrak{p}=\mathfrak{p}_1\oplus \mathfrak{p}_2 \oplus \cdots \oplus \mathfrak{p}_l$,
where $\mathfrak{p}_i$ are irreducible $\Ad(H)$-invariant submodules of $\mathfrak{p}$,
then for every positive real $x_i$, $i=1,\dots,l$, we get $\Ad(H)$-invariant inner products
\begin{equation}\label{hommetr}
x_1 \langle \cdot ,\cdot \rangle|_{\mathfrak{p}_1}+x_2 \langle \cdot ,\cdot \rangle|_{\mathfrak{p}_2}+
\cdots +x_l \langle \cdot ,\cdot \rangle|_{\mathfrak{p}_l}
\end{equation}
on the space $\mathfrak{p}$. Moreover, if the submodules $\mathfrak{p}_i$ are mutually inequivalent, then every
$\Ad(H)$-invariant inner product
$(\cdot,\cdot)$ on  $\mathfrak{p}$ has the above form (\ref{hommetr}).

\textit{For compact matrix groups $G$ we shall always use $\Ad(G)$-invariant inner product}
\begin{equation}
\label{innerprod}
\langle U , V \rangle = \frac{1}{2}\operatorname{Re}(\trace(UV^*)),\quad U,V\in \mathfrak{g},
\end{equation}
where $V^*=\overline{V}^T$ and $\overline{(\cdot)},$ $(\cdot)^T$ are operations of transposition and complex or
quaternionic conjugation, respectively. In all cases we will use representations by skew-Hermitian matrices and,
therefore, we may suppose that
$\langle U , V \rangle=-\frac{1}{2}\operatorname{Re}(\trace(UV))$.
\bigskip

The main tool for our goal is the following

\begin{theorem}[Theorem 8 in \cite{BerNik}]\label{body}
A compact Riemannian manifold $(G/H,g)$ is
$G$-generalized normal homogeneous if and only if there
exists an $\Ad(G)$-invariant centrally symmetric (relative to
zero) convex body $B$ in $\mathfrak{g}$ such that
$P(B)=\{v\in \mathfrak{p}\,|\, (v,v)\leq 1\}$,
where $P:\mathfrak{g}\rightarrow \mathfrak{p}$ is $\langle \cdot,\cdot
\rangle$-orthogonal projection.
\end{theorem}

We also need the notion of (dual) $2$-means by W.J.~Firey (see p. 18 of \cite{F2}).

The following two facts are known from H.~Minkowski. For any convex body $K$ with inner zero point in Euclidean
vector space $\mathbb{R}^n$, supplied with the standard inner product $(\cdot, \cdot)$,
its \textit{support function}
$$
h(u)=\max \{(u,v) \,|\, v\in K \}
$$
is nonzero nonnegative, convex, and positively homogeneous. This means that
$$h((1-\theta)u_1+\theta u_2)\leq (1-\theta)h(u_1)+\theta h(u_2),\quad \theta\in [0,1],\quad u_1,u_2\in \mathbb{R}^n,$$
$$h(\lambda u)=\lambda h(u),\quad \lambda \geq 0,\quad u\in \mathbb{R}^n.$$
Conversely, any real-valued nonzero nonnegative, convex, and positively homogeneous function on $\mathbb{R}^n$
is the support function of a unique convex body $K\subset \mathbb{R}^n$ with inner zero point.

Consider two convex bodies $K_1$ and $K_2$ in $\mathbb{R}^n$ with zero interior point and the support functions
$h_i;$ $i=1,2$. Let us fix any real numbers $p \geq 1$ and $\theta \in [0,1]$. For every nonnegative numbers
$a_1$ and $a_2$ define the value $M_p(a_1,a_2)=\bigl((1-\theta)a_1^p+\theta a_2^p\bigr)^{1/p}$.
Then (for every $p \geq 1$ and $\theta \in [0,1]$) the function $u \mapsto M_p(h_1(u), h_2(u))$ is a support
function of some convex body in $\mathbb{R}^n$ with inner zero point since it is clear that this function
is nonzero nonnegative, convex, and positively homogeneous. We shall  denote this body by
$K_{\theta}^{(p)}=K_{\theta}^{(p)}(K_1,K_2)$ (see details in \cite{F2}). We shall call this body
$K_{\theta}^{(p)}$ as {\it the dual $p$-mean of the bodies $K_1$ and $K_2$
(with parameter $\theta \in [0,1]$)}. For $\theta=1/2$ it is also called {\it a dual $p$-sum of $K_1$ and $K_2$}.

Let $L_m$ be an $m$-dimensional subspace of $\mathbb{R}^n$, $m<n$. Denote by $K^*$ the orthogonal projection
of a convex body $K \subset \mathbb{R}^n$ to $L_m$. It is clear that the support function of $K^*$ is the restriction
of the support function of $K$ to $L_m$. Therefore, we obtain
\begin{equation}\label{prpr}
\bigl(K_{\theta}^{(p)}(K_1,K_2)\bigr)^*=K_{\theta}^{(p)}(K_1^*,K_2^*)
\end{equation}
for every $p \geq 1$ and $\theta \in [0,1]$
(see also p. 22 of \cite{F2}).
\smallskip

Let $\{\cdot,\cdot\}$ be any inner product on $\mathbb{R}^n.$ There is unique positively
definite symmetric (with respect to $(\cdot, \cdot)$) linear operator
$A: \mathbb{R}^n\rightarrow \mathbb{R}^n$ such that $\{\cdot,\cdot\}=(\cdot, A \cdot).$

The following statement has been proved on pp. 53 and 54 from \cite{F3}.

\begin{prop}
\label{eucl}
Let $E$ be unit ball in $\{\mathbb{R}^n,(\cdot,\cdot)\}.$ Then the support function of $E$ is
\begin{equation}
\label{eucl1}
h(u)=\sqrt{(u,A^{-1}u )}.
\end{equation}
\end{prop}

\begin{proof}
It follows from definition that
$$h(u)=\max \{(u,v) \,|\, \langle v, Av\rangle =1\}.$$
Any such $v$ can be represented in the form $v=\lambda A^{-1}u+w,$ where $\lambda\in \mathbb{R}$ and
$(u, w)=0.$ Then
$$(v, Av)= \lambda^2 (u, A^{-1}u) + (w, Aw);$$
$$(u,v)=\lambda (u, A^{-1}u) = \sqrt{(u, A^{-1}u)}\sqrt{1-(w,Aw)},$$
so the maximum of $(u,v)$ is attained for $w=0$ and is equal to (\ref{eucl1}).
\end{proof}

\begin{prop}
\label{twoeucl}
Let $\{\cdot,\cdot\}_1$, $\{\cdot,\cdot\}_2$ be two inner products on $(\mathbb{R}^n, (\cdot,\cdot))$
with corresponding operators $A_1$ and $A_2$ and unit balls  $E_1$ and $E_2.$ Then $K_{\theta}^{(2)}(E_1,E_2)$ is the
unit ball of inner product $\{\cdot,\cdot\}=(\cdot, A \cdot ),$ where
\begin{equation}\label{gel}
A=\bigl((1-\theta)A_1^{-1}+\theta A_2^{-1} \bigr)^{-1}.
\end{equation}
\end{prop}

\begin{proof}
Let $h_1,$ $h_2,$ $h$ be the support functions for $E_1,$ $E_2,$ $K_{\theta}^{(2)}(E_1,E_2).$ Then by definition of
$K_{\theta}^{(2)}(E_1,E_2)$ and Proposition \ref{eucl},
\begin{eqnarray*}
h(u)=\sqrt{(1-\theta)h_1^2(u)+\theta h_2^2(u)}=\\
\sqrt{(1-\theta)(u, A_1^{-1}u) +\theta (u, A_2^{-1}u)}=
\sqrt{(u, [(1-\theta)A_1^{-1}+\theta A_2^{-1}]u)}.
\end{eqnarray*}
It is clear that $(1-\theta)A_1^{-1}+\theta A_2^{-1}$ is a positively definite and symmetric linear operator on
$(\mathbb{R}^n,(\cdot, \cdot))$ and there is the operator (\ref{gel}) with the same properties.
Now the statement follows from above calculations for $h$ and Proposition \ref{eucl}.
\end{proof}

\begin{theorem}\label{fireysum}
Suppose that two $G$-invariant metrics $g_1=\langle \cdot, A_1 \cdot \rangle$ and
$g_2= \langle \cdot, A_2 \cdot \rangle$ on a compact homogeneous space
$G/H$ are $G$-generalized normal homogeneous. Then every $G$-invariant metric
$g_{\theta}=\langle \cdot, A \cdot \rangle,$ where $A=[(1-\theta)A_1^{-1}+\theta A_2^{-1}]^{-1}$,
$\theta\in [0,1]$, is also a $G$-generalized normal homogeneous metric on $G/H$.
\end{theorem}

\begin{proof}
Fix a decomposition (\ref{reductivedecomposition}). Let $(\cdot,\cdot)_1$ and $(\cdot,\cdot)_2$
be $\Ad(H)$-invariant inner products on $\mathfrak{p}$ that generate $G$-invariant Riemannian metrics $g_1$ and $g_2$
respectively. Consider $\Ad(H)$-equivariant symmetric (with respect to $\langle \cdot, \cdot \rangle$)
positive definite
operators $A_1, A_2 :\mathfrak{p} \rightarrow \mathfrak{p}$, such that
$(u,v)_1=\langle u, A_1v \rangle$ and $(u,v)_2=\langle v, A_2v \rangle$ for all $u, v\in \mathfrak{p}$.

Further, we will use Theorem \ref{body}. By this theorem, there
are centrally symmetric (relative to zero) convex bodies $B_1$ and $B_2$ in $\mathfrak{g}$ such that
$P(B_1)=\{v\in \mathfrak{p}\,|\, (v,v)_1=\langle v, A_1v \rangle \leq 1\}:=E_1$ and
$P(B_2)=\{v\in \mathfrak{p}\,|\, (v,v)_2=\langle v, A_2v \rangle \leq 1\}=:E_2$.

Let us fix $\theta \in [0,1]$. Consider
the dual $2$-mean $K_{\theta}^{(2)}(B_1,B_2)$ in $\mathfrak{g}$
and the dual $2$-mean $K_{\theta}^{(2)}(E_1,E_2)$ in $\mathfrak{p}$.
It is clear that $P\bigl(K_{\theta}^{(2)}(B_1,B_2)\bigr)=K_{\theta}^{(2)}(E_1,E_2)$ by (\ref{prpr}).
On the other hand, by Proposition \ref{twoeucl}, $K_{\theta}^{(2)}(E_1,E_2)$ is unit ball $E$ of (clearly,
$\Ad(H)$- invariant) inner product $(u,v)=\langle u, Av \rangle$ on $\mathfrak{p},$ where $A$
is defined by formula (\ref{gel}). It is also clear that $K_{\theta}^{(2)}(B_1,B_2)$ is centrally
symmetric convex bodies in $\mathfrak{g}$,
because both $B_1$ and $B_2$ have this property. Therefore, again by Theorem \ref{body}, a
$G$-invariant metric $g_{\theta}$ on $G/H$, generated by the inner product $(\cdot, \cdot)$ on $\mathfrak{p}$,
is $G$-generalized normal homogeneous.
\end{proof}

Note that the inner product $\langle \cdot ,\cdot\rangle|_{\mathfrak{p}}$ generates a normal metric on $G/H$.
Any invariant Riemannian metric $g$ on $G/H$ is generated by some $\Ad(H)$-invariant inner product
$(\cdot, \cdot)=\langle \cdot, A \cdot \rangle|_{\mathfrak{p}}$ on $\mathfrak{p}$ that corresponds to a symmetric
positively definite operator
$A:(\mathfrak{p}, \langle \cdot ,\cdot\rangle|_{\mathfrak{p}}) \rightarrow
(\mathfrak{p}, \langle \cdot ,\cdot\rangle|_{\mathfrak{p}})$. Now, consider the inner product
$(\cdot, \cdot)'= \langle \cdot, A^{-1} \cdot \rangle.$ A $G$-invariant metric $\widehat{g}$ on $G/H$
that corresponds to this new inner product, we will call
{\it dual to the metric $g$ {\rm (}with respect to
$\langle \cdot, \cdot \rangle|_{\mathfrak{p}}$ {\rm)}}. It is easy to see that $\widehat{\widehat{g}}=g$,
$\widehat{\alpha g}={\alpha}^{-1}\widehat{g}$ for $\alpha >0$.

From Theorem \ref{fireysum} we immediately obtain

\begin{theorem}\label{fireysum2}
Let $GN$  be the set of all $G$-generalized normal homogeneous metrics
on a compact homogeneous space
$G/H$. Let $\widehat{GN}$ be the set of all $G$-invariant metrics dual to metrics from $GN$.
Then $\widehat{GN}$ is a convex cone in the cone of all $G$-invariant metrics on $G/H$.
\end{theorem}

We can identify $G$-invariant metrics on $G/H$ with corresponding inner products on $\mathfrak{p}$.
It is easy to see that the metric
$$
x_1^{-1} \langle \cdot ,\cdot \rangle|_{\mathfrak{p}_1}+x_2^{-1} \langle \cdot ,\cdot \rangle|_{\mathfrak{p}_2}+
\cdots +x_l^{-1} \langle \cdot ,\cdot \rangle|_{\mathfrak{p}_l}
$$
is dual to the metric (\ref{hommetr}).

From this and Theorem \ref{fireysum} we get

\begin{corollary}\label{cor1}
Suppose that both
$$
x_1\langle \cdot ,\cdot \rangle|_{\mathfrak{p}_1}+x_2 \langle \cdot ,\cdot \rangle|_{\mathfrak{p}_2}+
\cdots +x_l \langle \cdot ,\cdot \rangle|_{\mathfrak{p}_l}
$$
and
$$
y_1 \langle \cdot ,\cdot \rangle|_{\mathfrak{p}_1}+y_2 \langle \cdot ,\cdot \rangle|_{\mathfrak{p}_2}+
\cdots +y_l \langle \cdot ,\cdot \rangle|_{\mathfrak{p}_l}
$$
are $G$-generalized normal homogeneous on a compact space $G/H$. Then for any $\theta \in [0,1],$
the metric
$$
\bigl((1-\theta)x_1^{-1}+\theta y_1^{-1} \bigr)^{-1} \langle \cdot ,\cdot \rangle|_{\mathfrak{p}_1}+
\cdots +\bigl((1-\theta)x_l^{-1}+\theta y_l^{-1} \bigr)^{-1} \langle \cdot ,\cdot \rangle|_{\mathfrak{p}_l}
$$
has the same property.
\end{corollary}

As a simple partial case we have

\begin{corollary}\label{cor2}
Suppose that $l=2$ and metrics
$$
\alpha \langle \cdot ,\cdot \rangle|_{\mathfrak{p}_1}+\beta \langle \cdot ,\cdot \rangle|_{\mathfrak{p}_2}
\quad \mbox{and} \quad
\alpha \langle \cdot ,\cdot \rangle|_{\mathfrak{p}_1}+\gamma \langle \cdot ,\cdot \rangle|_{\mathfrak{p}_2},
$$
where $\alpha, \beta, \gamma >0$,
are $G$-generalized normal homogeneous on a compact space $G/H$. Then for any $r \in [0,1],$
the metric
$$
\alpha \langle \cdot ,\cdot \rangle|_{\mathfrak{p}_1}+
\bigl((1-r)\beta + r\gamma \bigr) \langle \cdot ,\cdot \rangle|_{\mathfrak{p}_2}
$$
has the same property.
\end{corollary}

\begin{proof}
It suffices to use Corollary \ref{cor1} and the fact that
the image of the function $\theta \in [0,1] \mapsto \bigl((1-\theta)\beta^{-1} +\theta \gamma^{-1} \bigr)^{-1}$
is the closed interval of $\mathbb{R}$ with the endpoints $\beta$ and $\gamma$.
\end{proof}

\begin{remark}
\label{theta}
One can easily check that one needs to take $\theta=r\gamma/[(1-r)\beta+r\gamma]$.
\end{remark}

\section{Killing vector fields of constant length on round spheres}

In recent papers \cite{BerNik7, BerNik6, BerNik5} the authors obtained various results, connected with
characterizations of Killing vector fields of constant length on Riemannian manifolds. The structure of the set
of such fields on some classical Riemannian manifolds is rather complicated. Here we shall discuss corresponding
questions for Euclidean spheres $S^{n}$.

We shall need a unified notation. By $\mathbb{F}$ we mean one of three classical fields: $\mathbb{R}$ of real
numbers, $\mathbb{C}$ of complex numbers, and $\mathbb{H}$ of quaternions with usual
inclusions $\mathbb{R}\subset \mathbb{C}\subset \mathbb{H}$. Since $\mathbb{H}$ contains
both $\mathbb{R}$ and $\mathbb{C}$ as subfields, it
will be convenient sometimes to represent elements of any mentioned fields as \textit{quaternions}. By $\mathbb{F}^n$
we denote a $n$-dimensional vector space over the field $\mathbb{F}$ (left in the case $\mathbb{F}=\mathbb{H}$);
we understand vectors as \textit{vectors-columns}. Now we suggest that the vector space $\mathbb{F}^{n+1}$
is supplied with the standard inner (scalar) product $(v, w)=\operatorname{Re}(\overline{v}^Tw)$, generating the
norm $\|\cdot\|$.

The Lie group of all linear transformations of $\mathbb{F}^{n+1}$ over $\mathbb{F}$, preserving the inner
product $(\cdot,\cdot)$, is denoted
by $U_{\mathbb{F}}(n+1)$. The elements of $U_{\mathbb{F}}(n+1)$ are represented by $((n+1)\times (n+1))$-matrices
$A$ over $\mathbb{F}$ with the property $A\overline{A}^T=AA^*=\Id_{n+1}$, where $\Id_{n+1}$ is unit diagonal
$((n+1)\times (n+1))$-matrix; $A\in U_{\mathbb{F}}(n+1)$ acts on $\mathbb{F}^{n+1}$ in the usual way:
if $x\in \mathbb{F}^{n+1}$, then $A(x)=Ax$. The Lie algebra $\mathfrak{u}_{\mathbb{F}}(n+1)$ of Lie groups
consists of matrices $U,$
subject to equation $U+\overline{U}^T=0,$ or $U+U^*=0.$ Thus $U_{\mathbb{R}}(n+1)=O(n+1),$ $U_{\mathbb{C}}(n+1)=U(n+1)$,
$U_{\mathbb{H}}(n+1)=Sp(n+1)$, and $\mathfrak{u}_{\mathbb{R}}(n+1)=\mathfrak{so}(n+1)$,
$\mathfrak{u}_{\mathbb{C}}(n+1)=\mathfrak{u}(n+1)$, $\mathfrak{u}_{\mathbb{H}}(n+1)=\mathfrak{sp}(n+1)$.

It is clear that the Lie group $U_{\mathbb{F}}(n+1)$ acts transitively and isometrically on the unit sphere
$$
S^{d(\mathbb{F})(n+1)-1}=\{x\in \mathbb{F}^{n+1} \,|\,\|x\|=1\},
$$
where $d(\mathbb{F})$ is the dimension of $\mathbb{F}$ over $\mathbb{R}$.
Here and later, elements of the Lie algebra $\mathfrak{u}_{\mathbb{F}}(n+1)$ are identified with the Killing vector
fields on $S^{d(\mathbb{F})(n+1)-1}$.

\begin{prop}\label{so}
The following two assertions are equivalent:

1) A Killing vector field $U \in \mathfrak{u}_{\mathbb{F}}(n+1)$ has constant length $C\geq 0$ on $S^{d(\mathbb{F})(n+1)-1}$;

2) $U^2=-C^2 \Id$.

\noindent If $\mathbb{F}=\mathbb{R}$ or $\mathbb{F}=\mathbb{C}$ then 1) and 2) are equivalent to

3) All eigenvalues of $U$ have the form $\pm \sqrt{C} {\bf i}$.

\noindent Moreover, the following assertions hold:

4) $U \in \mathfrak{u}_{\mathbb{F}}(n+1)$ is a unit Killing vector field on $S^{d(\mathbb{F})(n+1)-1}$
if and only if
\begin{equation}
\label{ufUf}
U\in \mathfrak{u}_{\mathbb{F}}(n+1)\cap U_{\mathbb{F}}(n+1).
\end{equation}

5) If $\mathbb{F}=\mathbb{R}$ or $U\in \mathfrak{su}(n+1)$, then 4) is equivalent respectively to the inclusion
$$U\in \mathfrak{so}(n+1)\cap SO(n+1)\quad\mbox{or}\quad U\in \mathfrak{su}(n+1)\cap SU(n+1),$$
where $n+1$ must be even in both cases.
\end{prop}

\begin{proof}
The value of the Killing vector field $U \in \mathfrak{u}_{\mathbb{F}}(n+1)$ at a point $x\in S^{d(\mathbb{F})(n+1)-1}$
is $U(x)=Ux$, and its length at this point is equal to $\|Ux\|$. It is clear that $U$ has constant length $C$
on $S^{d(\mathbb{F})(n+1)-1}$ if and only if $\|Ux\|=C\|x\|$ for all $x\in \mathbb{F}^{n+1}$.
This is equivalent to equalities
$$
C^2(x,x)=(Ux,Ux)=\operatorname{Re}(\overline{x}\overline{U}^TUx)=-(x,U^2x)=-(U^2x,x),
$$
where $U^2$ is a self-adjoint (symmetric) linear operator.
It is clear that the above equalities are equivalent to the equality
$U^2=-C^2 \Id$. So the assertions 1) and 2) are equivalent.

3) We can assume that $U$ is non-zero. Since $U$ is skew-Hermitian, 2) means that the matrix
$V:=\frac{1}{\sqrt{C}}U$ is \textit{orthogonal or unitary}. Hence, all eigenvalues of $V$ have unit norm.
On the other hand, $V$ is also \textit{skew-symmetric or skew-Hermitian} and all eigenvalues of $V$ should
be purely imaginary. Therefore, in these cases
1), 2), and 3) are pairwise equivalent.

4) Assertion 2) for $C=1$ has the form $\Id=-U^2=UU^*$ (since $U^*=-U \in \mathfrak{u}_{\mathbb{F}}(n+1)$),
that is equivalent to (\ref{ufUf}).

5) By statements 2) and 3), in both these cases $U$ is unit Killing vector field on $S^{d(\mathbb{F})(n+1)-1}$
if and only if all eigenvalues of $U$ have the form $\pm {\bf i},$ and additionally the sum of all these eigenvalues
is equal to zero. This proves all statements in 5).
\end{proof}

\smallskip

By the last statement in Proposition \ref{so}, \textit{nonzero} Killing vector fields $U \in \mathfrak{so}(n+1)$
(respectively $U\in \mathfrak{su}(n+1)$) of constant length on $S^{n}$ (respectively, on $S^{2n+1}$) can exist
only for odd $n$. Thus later we shall consider only odd-dimensional spheres $S^{2n+1}$. Note that the set of Killing
vector fields $U \in \mathfrak{so}(2(n+1))$ of constant length on odd-dimensional spheres $S^{2n+1}$ is very extensive
(see Proposition \ref{so} or the \cite{BerNik5} for a more detailed description).

\begin{theorem}\label{osn_sp}
Any unit sphere $S^{d(\mathbb{F})(n+1)-1},n\geq 1,$ is $U_{\mathbb{F}}(n+1)$-Clifford-Wolf homogeneous
{\rm (}in addition, for $\mathbb{F}=\mathbb{R}$, $n+1$ should be even {\rm )}.
\end{theorem}

\begin{proof}
By \cite{BerNik5}, it is enough to prove that for any tangent vector $v\in S^{d(\mathbb{F})(n+1)-1}_{x_0}$,
where $x_0\in S^{d(\mathbb{F})(n+1)-1},$ there exists a Killing vector field $K(x)=Ux$, $x\in S^{d(\mathbb{F})(n+1)-1}$,
where $U\in \mathfrak{u}_{\mathbb{F}}(n+1)$, of constant length on $S^{d(\mathbb{F})(n+1)-1},$ such that $Ux_0=v$. In view of
transitive action of the Lie group $G=U_{\mathbb{F}}(n+1)$ on $S^{d(\mathbb{F})(n+1)-1}$ by isometries one can assume
that
$x_0=(1,0,\dots,0)^T$. Then the Lie group $H=U_{\mathbb{F}}(n)$ is the isotropy subgroup of Lie group $G$ at the point
$x_0$, and one can consider $S^{d(\mathbb{F})(n+1)-1}$ as the homogeneous manifold
$(G/H=U_{\mathbb{F}}(n+1)/U_{\mathbb{F}}(n),\mu)$ with a suitable invariant Riemannian metric $\mu$.

One gets $\Ad(U_{\mathbb{F}}(n))$-invariant $\left\langle \cdot, \cdot \right\rangle$-orthogonal decomposition
\begin{equation}
\label{sum}
\mathfrak{u}_{\mathbb{F}}(n+1)=\mathfrak{p}\oplus \mathfrak{u}_{\mathbb{F}}(n)=
\mathfrak{p}_1\oplus \mathfrak{p}_2\oplus \mathfrak{u}_{\mathbb{F}}(n)=\mathfrak{p}_1\oplus \mathfrak{u}_{\mathbb{F}}(1)\oplus \mathfrak{u}_{\mathbb{F}}(n),
\end{equation}
where
\begin{eqnarray}
\mathfrak{p}_1&=&\left\{\left(\begin{array}{rr}
0 & -\overline{u}^T\\ u& 0_{nn}
\end{array} \right), u\in \mathbb{F}^n \right\},\label{p1}
\\
\mathfrak{p}_2&=&\left\{\left(\begin{array}{rr}
u_1 & 0_n^T\\ 0_n& 0_{nn}\end{array} \right), u_1\in \operatorname{Im}(\mathbb{F}) \right\}.\label{p2}
\end{eqnarray}
The low indices above indicate the size of zero block-matrices. Note that $\mathfrak{p}_2=\{0\}$ and $u_1=\{0\}$
for $\mathbb{F}=\mathbb{R}$.

In view of transitivity of Lie group $U_{\mathbb{F}}(n+1)$ on $S^{d(\mathbb{F})(n+1)-1},$ the tangent vector space
$S^{d(\mathbb{F})(n+1)-1}_{x_0}$ of $S^{d(\mathbb{F})(n+1)-1}$ at
the point $x_0$ is realized in the form $Ux_0,$ $U\in \mathfrak{u}_{\mathbb{F}}(n+1)$. In addition,
the correspondence
$$U\in \mathfrak{p}\rightarrow Ux_0=(u_1,u)^T\in S^{d(\mathbb{F})(n+1)-1}_{x_0}$$
defines an isomorphism of vector space $\mathfrak{p}$ onto vector space $S^{d(\mathbb{F})(n+1)-1}_{x_0}.$

Our aim is to find for every vector-matrix $X+Y\in \mathfrak{p}_1\oplus \mathfrak{p}_2$ a vector-matrix
$Z\in \mathfrak{u}_{\mathbb{F}}(n)$ of the form
\begin{equation}
\label{Z}
Z=\left\{\left(\begin{array}{rr}0 & 0_n^T\\ 0_n& U_{nn}\end{array} \right), U_{nn}\in \mathfrak{u}_{\mathbb{F}}(n) \right\},
\end{equation}
such that the Killing vector field $K(x):=Ux$, $x\in S^{d(\mathbb{F})(n+1)-1}$ on $S^{d(\mathbb{F})(n+1)-1}$
for the matrix $U:=X+Y+Z$ has constant length. Below we shall find such fields in different cases,
applying Proposition \ref{so}.

If $u\neq 0,$ we first consider the case

1) $u=(u_2,0,\dots,0)^T\in \mathbb{F}^n$, $u_2>0$.

a) If $\mathbb{F}=\mathbb{R}$, then $u_1=0,$ and it suffices to take as $U$ a block-diagonal matrix with the blocks
$\left(\begin{array}{rr}
0 & -u_2\\ u_2& 0
\end{array} \right)$
on the diagonal.

b) If $\mathbb{F}\neq \mathbb{R}$, it suffices to take as $U_{nn}$ a diagonal
$(n\times n)$-matrix with purely imaginary quaternions in $\mathbb{F}$ on the diagonal such that the element
in the left upper corner is equal to $-u_1$, and the squared modules of the others are equal to $u_2^2-u_1^2$.

2) Let us suppose that $u\in \mathbb{F}^n$, $|u|>0$. The group $U_{\mathbb{F}}(n)$ acts transitively on every sphere
in $\mathbb{F}^n$ with zero center. Therefore there exists an element $g\in U_{F}(n)$ such that
$g(|u|,0,\dots,0)^T=u$. In this case, instead of $U_{nn}$ above one needs to take the matrix
$U_{nn}'=\Ad(g)(U_{nn})$, where $U_{nn}$ is the same as above with $u_2:=|u|$.

The proof for the case $\mathbb{F}=\mathbb{R}$ is finished.

3) $u=0$, i.~e. $X=0$. In this case, it suffices to take as $U_{nn}$ (which determines the vector $Z$ by formula (\ref{Z})) a diagonal $(n\times n)$-matrix with arbitrary purely imaginary quaternions in $\mathbb{F}$ on the
diagonal, whose squared modules are equal to  $-u_1^2$.
\end{proof}

\begin{prop}\label{osn_u}
Any Euclidean sphere $S^{2n+1}$, $n\geq 1$, is $SU(n+1)$-Clifford-Wolf homogeneous for odd $n$.
\end{prop}

\begin{proof}
Proposition follows from the first inclusion in (\ref{incl}) and Theorem \ref{osn_sp}.
\end{proof}

\medskip

\section{On $\delta$-vectors and generalized normal homogeneity}

Let us suppose that $M=(G/H,\mu)$ be a compact homogeneous connected
Riemannian manifold with connected (compact) Lie group $G$. Let
$\mathfrak{g}=\mathfrak{h}\oplus \mathfrak{p}$,
$\langle \cdot,\cdot
\rangle$, and $(\cdot,\cdot)$ be the same as in Section \ref{gensec}.
We will need the following

\begin{prop}[Corollary 6 in \cite{BerNik}]\label{addsymm}
Let $M=(G/H,\mu)$ be $G$-generalized normal, and let $N_G(H)$ be a normalizer of $H$ in $G$.
Then the inner product $(\cdot, \cdot)$, generating $\mu$,
is $\Ad(N_G(H))$-invariant.
\end{prop}

We consider one of the equivalent characterizations of $\delta$-vectors (see details in Section 6 of \cite{BerNik}).
A vector $W\in \mathfrak{g}$ is called \textit{$\delta$-vector} on the
Riemannian homogeneous manifold $(M=G/H,\mu)$ if
$(W_{\mathfrak{p}},W_{\mathfrak{p}})\geq (\Ad(a)(W)|_{\mathfrak{p}},\Ad(a)(W)|_{\mathfrak{p}})$
for any $a\in G$.
The following fact is very important for our goals.

\begin{prop}[Proposition 14 in \cite{BerNik}]\label{gennorm}
A Riemannian homogeneous space $M=(G/H,\mu)$ is $G$-generalized normal homogeneous if and only if
for any $X\in \mathfrak{p}$ there is $Y\in \mathfrak{h}$ such that $X+Y$ is $\delta$-vector.
\end{prop}

For any $X\in \mathfrak{p}$ we consider the set
$W(X)=\{Y\in \mathfrak{h}\,,\, X+Y\,\,\mbox{is}\,\, \delta\mbox{-vector}\}$.
If $W(X)\neq \emptyset$, then $W(X)$ is a compact and convex subset of
$\mathfrak{h}$, and there is a unique $\widetilde{Y} \in W(X)$ such that
$\langle \widetilde{Y}, \widetilde{Y} \rangle \leq \langle {Y}, {Y} \rangle$ for all $Y\in W(X)$
(see Proposition 11 in \cite{BerNik}). We will denote such $\widetilde{Y} \in W(X)$ by $w(X)$.

\begin{lemma}\label{deltal1}
Consider any $X\in \mathfrak{p}$ with $W(X)\neq \emptyset$ and
$Z_{\mathfrak{h}}(X)=\{Z\in \mathfrak{h}\,,[Z,X]=0\}$. Then for any $Z\in Z_{\mathfrak{h}}(X)$
the equality $[Z,w(X)]=0$ holds.
\end{lemma}

\begin{proof}
Let us consider $a=\exp(tZ)\in H$ for small $t$.
It is clear that $\Ad(a)(X)=X$ and $\Ad(a)(w(X))\in W(X)$
(the set of $\delta$-vectors is invariant under the action of $\Ad(H):\mathfrak{g} \rightarrow \mathfrak{g}$).
But $\langle \Ad(a)(w(X)), \Ad(a)(w(X)) \rangle = \langle w(X),w(X) \rangle$ and, therefore,
$\Ad(a)(w(X))=w(X)$ by the discussion right before the lemma. On the other hand,
$\Ad(a)(w(X))=w(X)+t[Z,w(X)]+o(t)$ when $t\to 0$, hence $[Z,w(X)]=0$.
\end{proof}

\begin{corollary}\label{deltal2}
If $X\in \mathfrak{p}$ with $W(X)\neq \emptyset$ is such that
the Lie algebra $Z_{\mathfrak{h}}(X)$ has trivial center and $\rk(Z_{\mathfrak{h}}(X))=\rk(\mathfrak{h})$,
then $w(X)=0$, i.~e. $X$ is a $\delta$-vector. Moreover,
$(X,[U,[U,X]]_{\mathfrak{p}})+([U,X]_{\mathfrak{p}},[U,X]_{\mathfrak{p}})\leq 0$
for all $U\in \mathfrak{g}$.
\end{corollary}

\begin{proof}
Assume that $w(X)\neq \emptyset$. Note that $w(X)\not\in Z_{\mathfrak{h}}(X)$ since $[w(X),Z_{\mathfrak{h}}(X)]=0$ by
Lemma \ref{deltal1} and $Z_{\mathfrak{h}}(X)$ has trivial center. Therefore,
the Lie algebra $\mathbb{R}w(X)+Z_{\mathfrak{h}}(X)\subset \mathfrak{h}$
has rank $\rk(Z_{\mathfrak{h}}(X))+1=\rk(\mathfrak{h})+1$ that is impossible. Hence, $w(X)=0$.

In Proposition 3 of \cite{BerNik}, it is proved that for any $\delta$-vector $X+Y$
($X\in \mathfrak{p}$ and $Y\in \mathfrak{h}$) the inequality
$(X,[U,[U,X+Y]]_{\mathfrak{p}})+([U,X+Y]_{\mathfrak{p}},[U,X+Y]_{\mathfrak{p}})\leq 0$
holds for any $U\in \mathfrak{g}$. If we put $Y=0$ in this inequality,
then we get the last assertion of the corollary.
\end{proof}

Now we prove the following

\begin{theorem}\label{delta_su}
Any Euclidean sphere $S^{2n+1}, n\geq 1$, is
$SU(n+1)$-generalized homogeneous.
\end{theorem}

\begin{proof}
By Proposition \ref{gennorm} it is sufficient to prove that for any tangent vector
\linebreak
$v\in S^{2n+1}_{x_0}$, $x_0=(1,0,\dots,0)^T \in S^{2n+1}$
there exists a Killing vector field $K(x)=Ux$, $x\in S^{2n+1}$,
where $U\in  \mathfrak{su}(n+1)$, such that $Ux_0=v$ and which is a $\delta$-vector.

As in the proof of Theorem \ref{osn_sp}, we have
the group $G=U(n+1),$ transitive on the sphere $S^{2n+1}$ with the isotropy subgroup $H=U(n)$ at the point $x_0$. Therefore, one can consider $S^{2n+1}$ as homogeneous manifold $(G/H,\mu)$ with corresponding
invariant Riemannian metric $\mu$. We will use decomposition (\ref{sum}) and the notation from the proof
of Theorem \ref{osn_sp}.

Our goal is to find for every vector-matrix $X+Y\in \mathfrak{p}_1\oplus \mathfrak{p}_2$ a vector-matrix
$Z\in \mathfrak{u}(n)$ of the form (\ref{Z}),
such that $U:=X+Y+Z \in \mathfrak{su}(n+1)$ and the Killing vector field $K(x):=Ux$, $x\in S^{2n+1}$,
is a $\delta$-vector.

We characterize all $\delta$-vectors
$U \in \mathfrak{u}(n+1)$ for the point $x_0\in S^{2n+1}$. If $x\in S^{2n+1}$, then
$$
(Ux,Ux)=(U^*Ux,x)=-(U^2 x,x).
$$
A matrix $U \in \mathfrak{u}(n+1)$ is a $\delta$-vector for the point $x_0\in S^{2n+1}$
if and only if
$$
(Ux_0,Ux_0)=-(U^2 x_0,x_0)\geq -(U^2 x,x)=(Ux,Ux)
$$
for all $x\in S^{2n+1}$.
Since the matrix
$-U^2$ is symmetric and positively definite, we can reformulate this in the following form:
$$
-U^2=\diag(\lambda^2,B),
$$
where $\lambda \in \mathbb{R}$ and symmetric matrix $B$ is such that
the matrix $\lambda^2 \cdot \Id_{n}-B$ is positive semi-definite. We will use this characterization in the
last part of the proof.

Below we shall find such fields $U$ in two cases. In both these cases we should choose a matrix
$U_{nn}$ that defines the vector $Z$ by formula (\ref{Z}).
Denote by $B(u_1)$ a $(n\times n)$-matrix such that its $(1,1)$-entry
is equal to $-u_1$ and all other entries are zero.

1) Suppose that $u=(u_2,0,\dots,0)^T\in \mathbb{C}^n$, $u_2\geq 0$.
In this case we may take $U_{nn}=B(u_1)$. Then
$-U^2=\diag(|u_1|^2+|u_2|^2, |u_1|^2+|u_2|^2, 0,0,\dots,0)$.
Obviously, $U\in \mathfrak{su}(n+1)$.

2) Now, suppose that $u\in \mathbb{C}^n$ is arbitrary and consider $|u|=\sqrt{(u,u)}\geq 0$.
The group $U(n)$ acts transitively on every sphere in $\mathbb{C}^n$
with zero center. Therefore, there exists an element $g\in U(n)$ such that $g(|u|,0,\dots,0)^T=u$. In this case
it suffices to take the matrix $U_{nn}=\Ad(g)(B(u_1))$. Let us take $U=\Ad(\widetilde{g})(\widetilde{U})$, where
$\tilde{g}=\diag(1,g)\in U(n+1)$, and $\widetilde{U}$ is the matrix $U$ constructed in the previous case with
$u_2=|u|$. Since the group $\Ad(U(n))$ preserves the set of $\delta$-vectors, then
$U$ is really a $\delta$-vector.
It is also clear that $U\in \mathfrak{su}(n+1)$, since $\trace(U_{nn})=\trace(B(u_1))=-u_1$.
\end{proof}

\section{New examples}

Let $G/H$ be one of the following homogeneous spaces $SO(n+1)/SO(n)=S^n$, $G_2/SU(3)=S^6$, and $Spin(7)/G_2=S^7$.
In the first case $G/H$ is irreducible symmetric and in the other two cases it is
isotropy irreducible, therefore, in all these cases the homogeneous space $G/H$ admits only one (up to scaling)
$G$-invariant metric that is $G$-normal (see e.~g. Chapter 7 in \cite{Bes}).

Now, we consider $SU(n+1)$-invariant metrics on the  homogeneous sphere $S^{2n+1}=SU(n+1)/SU(n)$.
Consider $\Ad(SU(n))$-invariant $\left\langle \cdot, \cdot \right\rangle$-orthogonal decomposition
\begin{equation}
\label{sum_su}
\mathfrak{g}=\mathfrak{su}(n+1)=\mathfrak{p}\oplus \mathfrak{h}=
\mathfrak{p}_1\oplus \mathfrak{p}_2\oplus \mathfrak{h},
\end{equation}
where
$\mathfrak{p}_1$ is the same as in (\ref{sum}), $\mathfrak{h}$ consists of element of the form (\ref{Z}) with
$U_{nn}\in \mathfrak{su}(n),$ and
\begin{equation}
\label{p_su2} \mathfrak{p}_2=\left\{{\bf i}a \cdot
\diag(n,-1,\dots,-1), a\in \mathbb{R}\right\}.
\end{equation}
It should be noted that the module $\mathfrak{p}_1$ is $\Ad(SU(n))$ irreducible only for $n \geq 2$,
because $SU(1)$ is a trivial group.

Now we consider the family $\xi_t$ of $SU(n+1)$-invariant metrics on the  sphere $S^{2n+1}$
that correspond to the inner products
\begin{equation}
\label{mu_s_su}
(\cdot, \cdot)_t=\langle \cdot, \cdot \rangle |_{\mathfrak{p}_1}+
\frac{2nt}{n+1}\langle \cdot, \cdot \rangle |_{\mathfrak{p}_2}
\end{equation}
on $\mathfrak{p}$ for $t>0$.

\begin{theorem}\label{new_su_m}
The homogeneous Riemannian space $(S^{2n+1}=SU(n+1)/SU(n),\xi_t)$ is
$U(n+1)$-generalized normal homogeneous for all $t \in (0,1]$.
It is also $SU(n+1)$-generalized normal homogeneous for all $t \in [(n+1)/2n,1]$.
\end{theorem}

\begin{proof}
By Table 1, $\xi_t$ is $SU(n+1)$-normal homogeneous iff $t=\frac{n+1}{2n}$, and
$U(n+1)$-normal homogeneous iff $t<\frac{n+1}{2n}$. The metric $\xi_t$ has constant sectional
curvature 1 on $S^{2n+1}$ for $t=1.$ This and Theorem~\ref{osn_sp} imply that
$(S^{2n+1}=SU(n+1)/SU(n),\xi_1)$ is $U(n+1)$-generalized
normal homogeneous. The space $(S^{2n+1}=SU(n+1)/SU(n),\xi_t),$ $t<\frac{n+1}{2n}$, is
$U(n+1)$-normal homogeneous, hence $U(n+1)$-generalized normal homogeneous by Theorem \ref{body}.
It is also $SU(n+1)$-generalized normal homogeneous for $t=1$ by Theorem \ref{delta_su}
and for $t=\frac{n+1}{2n}$ by Theorem \ref{body}.

Now, the theorem follows directly from Corollary \ref{cor2}.
\end{proof}

\begin{remark}
It follows from equality $\mu_{1,s}=\xi_s$ and the second inclusion in (\ref{incl}) that Theorem \ref{new_sp_ms}
contains a stronger statement for odd $n$ than the first statement in Theorem \ref{new_su_m}.
\end{remark}

We can also consider representation of the Riemannian manifold $(S^{2n+1},\xi_t)$ as the homogeneous space
$U(n+1)/U(n)$ with the ($U(n+1)$-invariant) metric $\xi_t$ generated with the inner product
\begin{equation}\label{u_alg_dec}
(\cdot, \cdot)_t=\langle \cdot, \cdot \rangle |_{\mathfrak{p}_1}+
2t\langle \cdot, \cdot \rangle |_{\mathfrak{p}_2}
\end{equation}
where we have used decomposition (\ref{sum}) and $\langle \cdot, \cdot \rangle$ is the
$\Ad(U(n+1))$-invariant inner product (\ref{innerprod}) on the Lie algebra $u(n+1)$.

\begin{theorem}\label{u_main}
The homogeneous Riemannian manifold $(S^{2n+1},\xi_t)$ is
$U(n+1)$-gene\-ralized normal homogeneous if and only if $t \in (0,1]$.
\end{theorem}

\begin{proof}
It follows from Theorem \ref{new_su_m} that
the homogeneous Riemannian space
\linebreak
$(S^{2n+1}=SU(n+1)/SU(n),\xi_t)$ is
$U(n+1)$-generalized normal homogeneous for all $t \in (0,1]$.

Now, suppose that the metric $\xi_t$ is $U(n+1)$-generalized normal homogeneous
on $S^{2n+1}=U(n+1)/U(n)$.
By Proposition 22 in \cite{BerNik} we know that
for every $X \in
\mathfrak{p}_1$, $Y \in \mathfrak{p}_2$ holds the inequality
$$
x_1 \langle [[Y,X],X]_\mathfrak{h},[[Y,X],X]_\mathfrak{h} \rangle
\geq (x_2-x_1) \langle
[[Y,X],X]_{\mathfrak{p}_2},[[Y,X],X]_{\mathfrak{p}_2} \rangle,
$$
where $x_1=1$ and $x_2=2t$ (see (\ref{u_alg_dec})). If we take
$$
Y=\diag({\bf i},0,\dots,0)\quad \mbox{and}\quad
X=\diag \left( \left(
\begin{array}{cc}
0&{\bf i}\\{\bf i}&0\\
\end{array}
\right),0,\dots,0) \right),
$$
then $[[Y,X],X]=-2Y+2Z$, where
$Z=\diag(0,{\bf i},0,\dots,0) \in \mathfrak{h}$. Since $\langle
-2Y, -2Y \rangle= \langle 2Z, 2Z \rangle=4$, then we get $4 \geq
(2t-1)\cdot 4$, i.~e. $t\leq 1$.
\end{proof}

\begin{remark}
For $n=1$ the assertion of Theorem \ref{u_main} also follows from the results of paper \cite{BerNik2009ch},
where the authors proved (in particular) that all $U(2)$-generalized normal homogeneous metrics on $S^3$
are either $U(2)$-normal homogeneous, or $SU(2)$-normal homogeneous.
\end{remark}

\begin{theorem}\label{su_main_del}
The homogeneous Riemannian manifold $(S^{2n+1},\xi_t)$ is
$SU(n+1)$-gene\-ralized normal homogeneous if and only if $t\in [(n+1)/2n,1]$.
\end{theorem}

\begin{proof}
It follows from Theorem \ref{new_su_m} that
the homogeneous Riemannian space
\linebreak
$(S^{2n+1}=SU(n+1)/SU(n),\xi_t)$ is
$SU(n+1)$-generalized normal homogeneous for $t\in [(n+1)/2n,1]$.

Now, suppose that the metric $\xi_t$ is $SU(n+1)$-generalized normal homogeneous.
This metric is generated by the inner product (\ref{mu_s_su}).
For any non-trivial $X \in \mathfrak{p}_2$ we see that $Z_{\mathfrak{h}}(X)=\mathfrak{h}$ (see Lemma \ref{deltal1}).
By Corollary \ref{deltal2} we get that $X$ is $\delta$-vector and
$(X,[U,[U,X]]_{\mathfrak{p}})+([U,X]_{\mathfrak{p}},[U,X]_{\mathfrak{p}})\leq 0$
for all $U\in \mathfrak{su}(n+1)$.

Let us take $X={\bf i} \cdot \diag(n,-1,\dots,-1)\in \mathfrak{p}_2$
and any nonzero $U\in \mathfrak{p}_1$ as in (\ref{p1}) with
$\mathbb{F}=\mathbb{C}.$ Then
$$
0\neq [U,X]=(n+1) \left(\begin{array}{rr} 0 & {\bf i}
\overline{u}^T\\{\bf i}  u& 0_{nn}
\end{array} \right)\in \mathfrak{p}_1.
$$
Therefore
\begin{eqnarray*}
0\geq (X,[U,[U,X]]_{\mathfrak{p}})+([U,X]_{\mathfrak{p}},[U,X]_{\mathfrak{p}})=\\
\frac{2nt}{n+1}\langle X,[U,[U,X]]\rangle+\langle [U,X]_{\mathfrak{p}},[U,X]_{\mathfrak{p}}\rangle=\\
-\frac{2nt}{n+1}\langle [U,X],[U,X] \rangle+\langle [U,X],[U,X]\rangle=\\
\left(1-\frac{2nt}{n+1}\right)
\langle [U,X],[U,X] \rangle.
\end{eqnarray*}
Since $[U,X]\neq 0$, we get $0\geq 1-\frac{2nt}{n+1}$ that is equivalent to $t \geq (n+1)/2n$.

The inequality $t\leq 1$ follows from Theorem \ref{u_main}, since every
$SU(n+1)$-generalized normal homogeneous Riemannian manifold is
also $U(n+1)$-generalized normal homogeneous.
\end{proof}

Very special is the case $SU(2)=S^3$. There is a $6$-dimensional space of left-invariant
Riemannian metrics on $SU(2)$. But we have the following

\begin{prop}\label{su(2)n}
If a left-invariant metric $\mu$ on $S^3=SU(2)$ is $SU(2)$-generalized normal homogeneous, then
it is a metric of constant sectional curvature.
\end{prop}

\begin{proof}
By Proposition \ref{addsymm}, such left-invariant metric
$\mu$ should be bi-invariant.
\linebreak
Therefore, $(S^3,\mu)$
has constant curvature, since there are only metrics of constant
\linebreak
curvature
among invariant Riemannian metrics
on the homogeneous space
\linebreak
$SU(2)\times SU(2)/\diag(SU(2))=SO(4)/SO(3)$.
\end{proof}

\bigskip

Now we consider metrics $\mu_t$ on $S^{4n+3}=Sp(n+1)/Sp(n)$. Such metrics are generated by inner products
\begin{equation}
\label{mut}
(\cdot,\cdot)_{t}:= \left\langle \cdot, \cdot \right\rangle|_{\mathfrak{p}_1}+
2t\left\langle \cdot, \cdot \right\rangle|_{\mathfrak{p}_{2}},
\end{equation}
where we have used
the $\Ad(Sp(n))$-invariant $\left\langle \cdot, \cdot \right\rangle$-orthogonal decomposition
\linebreak
$\mathfrak{sp}(n+1)=\mathfrak{p}\oplus \mathfrak{sp}(n)= \mathfrak{p}_1\oplus \mathfrak{p}_2\oplus \mathfrak{sp}(n)$, and
the modules $\mathfrak{p}_1$ and $\mathfrak{p}_2=\mathfrak{sp}(1)$ are defined
by formulas (\ref{p1}) and (\ref{p2})  (see (\ref{sum})).

\begin{theorem}\label{new_sp_m}
The homogeneous Riemannian space $(S^{4n+3}=Sp(n+1)/Sp(n),\mu_t)$
\linebreak
is
$Sp(n+1)\times Sp(1)$-generalized normal homogeneous for all $t \in (0,1]$
and
\linebreak
$Sp(n+1)$-generalized normal homogeneous for all $t \in [1/2,1]$.
\end{theorem}

\begin{proof}
The metric $\mu_t$ is defined by $\Ad(Sp(n))$-invariant inner product (\ref{mut}).
By Theorem \ref{osn_sp}, the sphere $(S^{4n+3},\mu_1)$ is
$Sp(n+1)$-generalized normal homogeneous, and hence $Sp(n+1)\times Sp(1)$-generalized normal homogeneous.
By Table 1 $(S^{4n+3},\mu_t)$ is $Sp(n+1)\times Sp(1)$-normal homogeneous for all
$t \in (0,1/2)$ and $Sp(n+1)$-normal homogeneous for $t=1/2$. The theorem follows from Theorem \ref{body}
and Corollary \ref{cor2}.
\end{proof}

\begin{theorem}\label{sp_sp_main}
The homogeneous Riemannian space $(S^{4n+3}=Sp(n+1)/Sp(n),\mu_t)$ is
$Sp(n+1)\times Sp(1)$-generalized normal homogeneous if and only if $t \in (0,1]$.
\end{theorem}

\begin{proof}
The metric $\mu_t$ is $Sp(n+1)\times Sp(1)$-generalized normal homogeneous for $t \in (0,1]$
by Theorem \ref{new_sp_m}.

Now, let $(S^{4n+3}=Sp(n+1)/Sp(n),\mu_t)$ be
$Sp(n+1)\times Sp(1)$-generalized normal homogeneous.
We have $\Ad(Sp(n))$-invariant $\left\langle \cdot, \cdot \right\rangle$-orthogonal decomposition
$$
\mathfrak{sp}(n+1)=\mathfrak{p}\oplus \mathfrak{sp}(n)= \mathfrak{p}_1\oplus \mathfrak{p}_2\oplus \mathfrak{sp}(n);\quad
\mathfrak{p}_2=\mathfrak{sp}(1),
$$
(see (\ref{sum})). Now we consider $\Ad(Sp(n)\times Sp(1))$-invariant decomposition
\begin{equation}\label{spspsum}
\mathfrak{sp}(n+1)\oplus \mathfrak{sp}(1)=\overline{\mathfrak{p}}\oplus \mathfrak{sp}(n)\oplus \diag(\mathfrak{sp}(1)),
\end{equation}
where (we identify elements $(X,0)\in \mathfrak{sp}(n+1)\oplus \mathfrak{sp}(1)$ with $X\in \mathfrak{sp}(n+1)$)
$$
\mathfrak{sp}(n)\oplus \diag(\mathfrak{sp}(1))\subset
(\mathfrak{sp}(1)\oplus \mathfrak{sp}(n))\oplus \mathfrak{sp}(1)\subset \mathfrak{sp}(n+1)\oplus \mathfrak{sp}(1),
$$
$$
\overline{\mathfrak{p}}=\mathfrak{p}_1\oplus \overline{\mathfrak{p}}_2,
\quad
\overline{\mathfrak{p}}_2=\{(X,-X)\,,\, X\in \mathfrak{sp}(1) \}\subset
\mathfrak{p}_2\oplus \mathfrak{sp}(1).
$$
It should be noted that any vector $(X,-X)\in \overline{\mathfrak{p}}_2$ is projected to the vector
$2X\in \mathfrak{p}_2$ when we project $\mathfrak{sp}(n+1)\oplus \mathfrak{sp}(1)$ to $\mathfrak{sp}(n+1)$ along
$\diag(\mathfrak{sp}(1))$.
Elements of the Lie algebra
$\mathfrak{sp}(n+1)\oplus \mathfrak{sp}(1)$ we consider as $\bigl((n+2)\times(n+2)\bigr)$-matrices.
Then
the metric $\mu_t$ is generated by the inner product
\begin{equation}\label{spsosum2}
(\cdot, \cdot)=\langle \cdot, \cdot \rangle |_{\mathfrak{p}_1}+
4t\langle \cdot, \cdot \rangle |_{\overline{\mathfrak{p}}_2}
\end{equation}
on $\overline{\mathfrak{p}}$ (compare with (\ref{mut})).

By Proposition 22 in \cite{BerNik} we know that
for every $X \in
\mathfrak{p}_1$, $Y \in \overline{\mathfrak{p}}_2$ the inequality
$$
x_1 \langle [[Y,X],X]_\mathfrak{h},[[Y,X],X]_\mathfrak{h} \rangle
\geq (x_2-x_1) \langle
[[Y,X],X]_{\overline{\mathfrak{p}}_2},[[Y,X],X]_{\overline{\mathfrak{p}}_2} \rangle.
$$
holds, where $x_1=1$, $x_2=4t$ and $\mathfrak{h}=\mathfrak{sp}(n)\oplus \diag(\mathfrak{sp}(1))$.
Now, if we take
$$
X=\left(\diag \left( \left(
\begin{array}{cc}
0&1\\-1&0\\
\end{array}
\right),0,\dots,0\right), \diag(0,0,\dots,0) \right)\in \mathfrak{p}_1
$$
and $Y=(\diag({\bf i},0,\dots,0),-\diag({\bf i},0,\dots,0))\in \overline{\mathfrak{p}}_2$,
then we get
$$[[Y,X],X]=-U_1-U_2+2U_3, \quad \mbox{where}$$
\begin{eqnarray*}
U_1&=&(\diag({\bf i},0,\dots,0),-\diag({\bf i},0,\dots,0))\in \overline{\mathfrak{p}}_2,\\
U_2&=&(\diag({\bf i},0,\dots,0),\diag({\bf i},0,\dots,0))\in \mathfrak{h},\\
U_3&=&(\diag(0,{\bf i},0,\dots,0),\diag(0,0,\dots,0))\in \mathfrak{sp}(n)\subset \mathfrak{h}.
\end{eqnarray*}
Since $\langle U_1,U_1 \rangle=\langle U_2,U_2 \rangle=1$, $\langle U_3,U_3 \rangle=1/2$ and $\langle U_2,U_3 \rangle=0$
we get
$$
1\cdot (1+4\cdot 1/2) \geq (4t-1) \cdot 1
$$
that proves the last assertion of the theorem.
\end{proof}

\begin{theorem}\label{new_sp_ms}
The homogeneous Riemannian space $(S^{4n+3}=Sp(n+1)/Sp(n),\mu_{t,s})$ for $s\neq t$ is
$Sp(n+1)\times U(1)$-generalized normal homogeneous for all $t \in [1/2,1]$ and $s\in (0,t)$.
\end{theorem}

\begin{proof}
The metric $\mu_{t,s}$ is defined by $\Ad(Sp(n))$-invariant inner product
\begin{equation}
\label{muts}
(\cdot,\cdot)_{t,s}:= \left\langle \cdot, \cdot \right\rangle|_{\mathfrak{p}_1}+
2t\left\langle \cdot, \cdot \right\rangle|_{\mathfrak{p}_{2,1}}+
2s\left\langle \cdot, \cdot \right\rangle|_{\mathfrak{p}_{2,2}}
\end{equation}
on $\mathfrak{p}$ for $\Ad(Sp(n))$-invariant $\langle \cdot,\cdot \rangle$-orthogonal decomposition, which is defined
by decomposition (\ref{sum}) and formulas (\ref{p1}), (\ref{p2}), with additional subdivision
\linebreak
$\mathfrak{p}_2=\mathfrak{p}_{2,1} \oplus \mathfrak{p}_{2,2}$. Formula (\ref{p2}) shows that any element $U\in \mathfrak{p}_2$
has a form of $((n+1)\times (n+1))$-matrix with unique non-zero element $U_{11}=u_1\in \operatorname{Im}(\mathbb{H})$ if
$U\neq 0$. Similarly, any element $U$ in the subspace $\mathfrak{p}_{2,1}\subset \mathfrak{p}_{2}$ (respectively,
$\mathfrak{p}_{2,2}\subset \mathfrak{p}_{2}$) is defined uniquely by $U_{11}=u_1\in \mathbb{R}{\bf j}\oplus \mathbb{R}{\bf k}$
(respectively, $U_{11}=u_1\in \mathbb{R}{\bf i}$).

By Theorem \ref{osn_sp}, the sphere $(S^{4n+3},\mu_{t,s})$ for $s=t=1$ is $Sp(n+1)$-generalized normal homogeneous,
and hence $Sp(n+1)\times U(1)$-generalized normal homogeneous. By Table~1 $(S^{4n+3},\mu_{t,s})$
is $Sp(n+1)\times U(1)$-normal homogeneous iff $0< s < t=1/2$. Hence, by Theorem \ref{body}, it is
$Sp(n+1)\times U(1)$-generalized normal homogeneous if $t=1/2$ and $0 <s <t$. So we can assume that $1/2 < t$. If
$t=1$, the statement follows from Corollary \ref{cor2}.

So we can suppose that we are given arbitrary pair $(t,s)$, where $1/2 < t < 1$ and $0 < s < t$.
Then $t=(1-r)\frac{1}{2}+r\cdot1$ for $r=2t-1.$ It follows from Remark \ref{theta} that
$t=\bigl((1-\theta)(\frac{1}{2})^{-1} +\theta 1^{-1} \bigr)^{-1}$ for $\theta=(2t-1)/t$. Now it is enough to prove
that there exists $s_1\in (0,1/2)$ such that $s=\bigl((1-\theta)s_1^{-1} +\theta 1^{-1} \bigr)^{-1}$ for
$\theta=(2t-1)/t$,
i.~e.
$$s=\left(\frac{(1-t)}{ts_1} + \frac{2t-1}{t}\right)^{-1}=\frac{ts_1}{(1-t)+(2t-1)s_1}.$$
From this equality it is not difficult to find that
$$s_1=\frac{s(1-t)}{t-(2t-1)s}.$$
It follows from conditions for the pair $(t,s)$ that both the numerator and the denominator of the above fraction are
positive, so $s_1>0.$ Now the inequality $s_1 < 1/2$ is equivalent to inequality $2(1-t)s < t-(2t-1)s$ or
$s < t$, which is satisfied by conditions of the theorem.
\end{proof}

We need the following general proposition.

\begin{prop}
\label{subm1}
Let $p: (M,\mu)\rightarrow (N,\nu)$ be a Riemannian submersion, which is a homogeneous fibration with respect to some
isometry Lie group $G$ of the space $(M,\mu),$ and the space $(M,\mu)$ is $G$-$\delta$-homogeneous. Then $(N,\nu)$ is
also $G$-$\delta$-homogeneous.
\end{prop}

\begin{proof}
Let $\rho_M$ and $\rho_N$ be the inner metrics on $M$ and $N$ (induced by the metric tensors $\mu$ and $\nu$).
Consider any points $x, y $ in $N$. In view of homogeneity, the space $(M,\rho_M)$ is finitely compact.
Then there exist points $\tilde{x}\in p^{-1}(x),$ $\tilde{y}\in p^{-1}(y)$ such that
$\rho_M(\tilde{x},\tilde{y})=\rho_N(x,y).$ Since $(M,\rho_M)$ is $G$-$\delta$-homogeneous there is some
$\delta(\tilde{x})$-translation $\tilde{g}\in G$ of the space $(M,\rho_M)$ such that
$\tilde{g}(\tilde{x})=\tilde{y}.$ As far as $G$ preserves the fibers of Riemannian submersion $p$ then
$\tilde{g}$ induces some isometry $g$ of the space  $(N,\rho_N).$ Let $z$ be an arbitrary point in $M$, $\tilde{z}$
is any point in the fiber $p^{-1}(z).$ Then $g(x)=y$ and
$$\rho_N(x,g(x))=\rho_N(x,y)=\rho_M(\tilde{x},\tilde{y})=\rho_M(\tilde{x},\tilde{g}(\tilde{x}))\geq$$
$$\rho_M(\tilde{z},\tilde{g}(\tilde{z}))\geq \rho_N(p(\tilde{z}),p(\tilde{g}(\tilde{z})))=\rho_N(z,g(z)).$$
Therefore $(N,\nu)$ $G$-$\delta$-homogeneous.
\end{proof}

The following proposition is a partial case of Proposition \ref{subm1}.

\begin{prop}
\label{subm}
Let $(M=G/H,\mu)$ and $(M_1=G/H_1,\nu)$ be homogeneous connected compact Riemannian manifolds,
$H\subset H_1,$ and the canonical projection $p: (M,\mu)\rightarrow (M_1,\nu),$
induced by the inclusion $H\subset H_1,$ is a Riemannian submersion. Then the space $(G/H_1,\nu)$
is $G$-generalized normal homogeneous if the space $(G/H,\mu)$ is $G$-generalized normal homogeneous.
\end{prop}

\begin{corollary}
\label{cp}
The space $(\mathbb{C}P^{2n+1}=Sp(n+1)/Sp(n)\cdot U(1),\nu_t)$ is $Sp(n+1)$-generalized normal homogeneous for
$t \in [1/2, 1].$
\end{corollary}

\begin{proof}
As it was said in section \ref{invest}, the Hopf fibration (\ref{pr1}) (also for $\mu_t=\mu_{t,s=t}$)
is a Riemannian submersion. Now the statement follows from Theorem \ref{new_sp_m} and Proposition~\ref{subm}.
\end{proof}

We will need the following
\begin{prop}[\cite{Al68}]\label{NWAL}
Let $(M=G/H, \mu)$ be any homogeneous Riemannian manifold and $T$
be any torus in $H$, $C(T)$ is its centralizer in $G$. Then
the orbit $M_T=C(T)(eH)$ is a totally geodesic submanifold of $(M, \mu)$.
\end{prop}

\begin{theorem}
\label{sp_u_main}
The Riemannian manifold $(S^{4n+3},\mu_{t,s})$ {\rm(}in particular, $\mu_{t,t}=\mu_t${\rm)} is
$Sp(n+1)\times U(1)$-generalized normal homogeneous if and only if $t \in[1/2,1]$ and $s\in (0,t]$.
\end{theorem}

\begin{proof}
It follows from Theorem \ref{new_sp_ms} that
the homogeneous Riemannian space $(S^{4n+3}=Sp(n+1)/Sp(n),\mu_{t,s})$ is
$Sp(n+1)\times U(1)$-generalized normal homogeneous for all $t \in [1/2,1]$ and $s\in (0,t)$.
The space $(S^{4n+3}=Sp(n+1)/Sp(n),\mu_t=\mu_{t,s})$ for $s=t$ is even
$Sp(n+1)$-generalized normal homogeneous for all $t \in [1/2,1]$ by Theorem \ref{new_sp_m}
(the assertion for the case $s=t$ could be also easily obtained by passing to the limit).

Now, let the Riemannian space $(S^{4n+3}=Sp(n+1)/Sp(n),\mu_{t,s})$ be
$Sp(n+1)\times U(1)$-generalized normal homogeneous for some $t>0$ and $s>0$.

Suppose that $t \notin [1/2,1]$.
Clearly, (\ref{pr1}) is a homogeneous fibration with respect to $Sp(n+1)\times U(1),$ and also a Riemannian
submersion, as it was said in Section \ref{invest}. Moreover, the subgroup
$\Id\times U(1)\subset Sp(n+1)\times U(1)$ induces trivial action on the base
$\mathbb{C}P^{2n+1}$ of the fibration (\ref{pr1}). Then $(\mathbb{C}P^{2n+1}, \nu_t)$ is
$Sp(n+1)$-generalized normal homogeneous by Proposition \ref{subm1}. But this, together with (\ref{muts}),
contradicts Proposition 28 in \cite{BerNik}. Therefore, $t \in[1/2,1]$.

Now, let us prove that $s\leq t$. For this we consider $S^{4n+3}=G/H$, where $G=Sp(n+1)\times U(1),$
$H=Sp(n)\times U(1)$, and the embeddings of $Sp(1)$ and $Sp(n)$ in $G$
are defined by the symmetric pair $(Sp(n+1), Sp(1)\times Sp(n))$, the embedding of $U(1)\subset H$ in $G$ is diagonal:
$a\mapsto (a,a)\subset Sp(1)\times U(1) \subset G$. Let $T$ be a maximal torus in $H$, consider its centralizer
$C(T)$ in $G$. It is easy to see that $U(1)\times U(1)\in C(T)$ and $Sp(1)\in C(T)$. Moreover,
the orbit $M_T=C(T)(eH)$ with induced Riemannian metrics is totally geodesic in $(S^{4n+3},\mu_{t,s})$ by Proposition
\ref{NWAL} and, therefore, is generalized normal homogeneous itself by Theorem~11 in \cite{BerNik}.
But it is easy to see, that this orbit is isometric to the Riemannian space $(U(2)/U(1), \mu)$,
where $\mu$ is generated by the inner product
$t\langle \cdot, \cdot \rangle |_{\mathfrak{p}_1}+
2s\langle \cdot, \cdot \rangle |_{\mathfrak{p}_2}$
(we have used decomposition (\ref{sum}) and $\Ad(U(2))$-invariant inner product (\ref{innerprod})
$\langle \cdot, \cdot \rangle$
on the Lie algebra $u(2)$).
Using a homothety and Theorem \ref{u_main}, we get $s\leq t$.
\end{proof}

\begin{theorem}
\label{sp_main}
The Riemannian manifold $(S^{4n+3},\mu_{t})$ is
$Sp(n+1)$-generalized normal homogeneous if and only if $t \in[1/2,1]$.
\end{theorem}

\begin{proof}
It follows from Theorem \ref{new_sp_m} that
the homogeneous Riemannian space $(S^{4n+3}=Sp(n+1)/Sp(n),\mu_t)$ is
$Sp(n+1)$-generalized normal homogeneous for all $t \in [1/2,1]$.
Let us suppose that $(S^{4n+3},\mu_{t})$ is $Sp(n+1)$-generalized normal homogeneous.
Then we have the canonical Riemannian submersion
$$pr_1: (S^{4n+3}=Sp(n+1)/Sp(n),\mu_{t})\rightarrow (\mathbb{C}p^{2n+1}=Sp(n+1)/U(1)\cdot Sp(n),\nu_t),$$
and by Proposition \ref{subm}, $(\mathbb{C}P^{2n+1},\nu_t)$ is also $Sp(n+1)$-generalized normal homogeneous.
Hence, as in the proof of Theorem \ref{sp_u_main}, we must have $t \in[1/2,1]$.
\end{proof}

\section{The spaces of unit Killing vector fields on spheres}

Let $\mathfrak{g}$ be the Lie algebra of a Lie group $G$ acting transitively on some sphere $S^n$ and
let $UKVF(\mathfrak{g},n)$ be the set of all unit Killing vector fields, lying in $\mathfrak{g}$.
The set $UKVF(\mathfrak{g},n)$, supplied
with the induced topology from $\mathfrak{g}$, becomes a topological space.
It is interesting to find topological structure of these spaces. Obviously,
$UKVF(\mathfrak{g},n)=\emptyset$ for any $\mathfrak{g}$ if $n$ is even, and $UKVF(\mathfrak{\mathfrak{so}}(2),1)$
has exactly two points. In this and in the last sections, we shall study  the spaces $UKVF(\mathfrak{g},n)$
for spheres $S^n$ with odd $n\geq 3$ and some connected transitive Lie groups $G$ on $S^n$.

\begin{prop}
\label{int1}
The space $O(2n)/U(n)$ (with symmetric space $SO(2n)/U(n)$ as a connected component)
can be considered as the space $UKVF(\mathfrak{so}(2n),2n-1)$ for $n\geq 1$.
\end{prop}

\begin{proof}
By Proposition 10 in \cite{BerNik5}, the space $UKVF(\mathfrak{so}(2n),2n-1)$ is a union of two
disjoint orbits with respect to the adjoint action of $SO(2n)$ and one orbit with respect to the adjoint action of $O(2n)$. As a unit Killing vector field on $S^{2n-1},$ one can take the matrix $U=\diag(C,...,C)\in SO(2n)\cap \mathfrak{so}(2n)$, where $C$ is one of two elements
in $SO(2)\cap so(2)$ (compare with Proposition \ref{so}). It is clear that $U\in U(n)\cap \mathfrak{u}(n)$ and $U$ lies in the center of $U(n)$ as well as in the center
of $\mathfrak{u}(n)$. Then the centralizer of $U$ in $O(2n)$, and therefore,
the stabilizer of $U\in \mathfrak{so}(2n)$ relative to adjoint action of $O(2n)$ is exactly $U(n)$
(see e.~g. Examples 22 and 23 in Section 1.2 of book \cite{On} by A.L.~Onishchik).
This implies the statement of proposition.
\end{proof}

\begin{remark}
As it was stated in \cite{Bes}, the symmetric space $SO(2n)/U(n)$ is the space of all
complex structures on $\mathbb{R}^{2n}$, compatible with standard Euclidean structure or is the space of metric-compatible fibrations $S^1\rightarrow \mathbb{R}P^{2n-1}\rightarrow \mathbb{C}P^{n-1}$.
\end{remark}

The spheres $S^{2n+1}$ are Clifford-Wolf homogeneous even with respect to $U(n+1)$ and
spheres $S^{4n+3}$ are also Clifford-Wolf homogeneous with respect to $SU(2(n+1))$ and $Sp(n+1)$
(see Proposition \ref{osn_u} and Theorem \ref{osn_sp}). We shall describe in this section the spaces
$UKVF(\mathfrak{u}(n+1),2n+1)$, $UKVF(\mathfrak{su}(n+1),2n+1)$, $UKVF(\mathfrak{sp}(n+1), 4n+3).$
Let $G$ be one of the groups $U(n+1)$, $SU(2(k+1))$ and $Sp(n+1)$. Then the Lie group $G$ acts by conjugation
on the space $L$ of all unit Killing vector fields from $\mathfrak{g}$ on corresponding sphere.
Therefore, $L$ is a disjoint union of some orbits of $\Ad(G)$.
In what follows, we will use the results of Section H of Chapter 8 in \cite{Bes}, where
one can find a detailed description of the structure of adjoint orbits that we need.

\begin{prop}
\label{uhlensp}
The K\"ahler symmetric space $Sp(n+1)/U(n+1)$ could be interpreted as the space $UKVF(\mathfrak{sp}(n+1),4n+3),$
where $\mathfrak{sp}(n+1)\subset \mathfrak{so}(4(n+1))$.
\end{prop}

\begin{proof}
It follows from Proposition \ref{so} and Example 8.116 in \cite{Bes} that any matrix
$U\in UKVF(\mathfrak{sp}(n+1),4n+3)$ is in $\Ad(Sp(n+1))$-orbit of a matrix of the type
$\widetilde{U}=\diag({\bf i}, {\bf i}, \dots, {\bf i})$. Hence, by the same Example 8.116 in \cite{Bes},
$$
UKVF(\mathfrak{sp}(n+1),4n+3)=\Ad(Sp(n+1))\bigl(\diag({\bf i}, {\bf i} , \dots, {\bf i})\bigr)=Sp(n+1)/U(n+1).
$$
\end{proof}

\begin{remark}
Note that $Sp(n+1)/U(n+1)$ could be identified with a manifold of totally isotropic complex
$(n+1)$-dimensional subspaces in $\mathbb{C}^{2(n+1)},$ (see e.g. 8.116 and 8.86 in \cite{Bes}) or
the set of all planes $\mathbb{C}P^{n+1}$ in $\mathbb{H}P^{n+1}$ or the set of all complex
structures in $\mathbb{H}^{n+1}$ (see e.g. Table 10.125 in \cite{Bes}).
\end{remark}

\begin{prop}
\label{uhlensu}
The complex Grassmannian $SU(2(k+1))/S(U(k+1)\times U(k+1))$ could be interpreted as the space
$UKVF(\mathfrak{su}(2(k+1)),4k+3),$
where $\mathfrak{su}(2(k+1))\subset \mathfrak{so}(4(k+1)$.
\end{prop}

\begin{proof}
It follows from Proposition \ref{so} and Example 8.111 in \cite{Bes} that any matrix
$U\in UKVF(\mathfrak{su}(2(k+1)),4k+3)$ is in $\Ad(SU(2(k+1)))$-orbit of a matrix of the type
$\widetilde{U}=\diag(\underbrace{{\bf i},  \dots, {\bf i}}_{k+1}\,, \underbrace{{\bf -i},  \dots, {\bf -i}}_{k+1}\,)$. Then
$$
UKVF(\mathfrak{su}(2(k+1)),4k+3)=SU(2(k+1))/S(U(k+1)\times U(k+1)),
$$
since $S(U(k+1)\times U(k+1))$ is the centralizer of $\widetilde{U}$ in $SU(2(k+1)).$
\end{proof}

The description of  the space $UKVF(\mathfrak{u}(n+1),2n+1)$ is a little more complicated.

\begin{prop}
\label{uhlenu} The space $UKVF(\mathfrak{u}(n+1),2n+1),$ $\mathfrak{u}(n+1)\subset \mathfrak{so}(2(n+1))$
could be naturally identified with the following disjoint union of complex Grassmannians
\begin{equation}
\label{orb}
\bigcup_{i=0}^{n+1} SU(n+1)/S(U(i)\times U(n+1-i)).
\end{equation}
\end{prop}

\begin{proof}
Since $\mathfrak{u}(n+1)=\mathbb{R} \oplus \mathfrak{su}(n+1)$, then for any $U \in \mathfrak{u}(n+1)$ we get
$U=U_1+U_2$,
where $U_1 \in \mathbb{R}$ and
$U_2\in \mathfrak{su}(n+1)$. By Example 8.111 in \cite{Bes}, $U_2$ lies in an
$\Ad(U(n+1))$-orbit (or, equivalently, $\Ad(SU(n+1))$-orbit) of a
matrix of the type
$\widetilde{U}=\diag({\bf i} \lambda_1, {\bf i} \lambda_2, \dots, {\bf i} \lambda_{n+1})$, where $\lambda_1
\geq \lambda_2\geq \cdots \geq \lambda_{n+1}$ and $\sum_{i=1}^{n+1} \lambda_i =0$. Note also that
$U_1=\alpha  \diag ({\bf i}, {\bf i},\dots, {\bf i})$ for some
$\alpha \in \mathbb{R}$ and $\Ad(a)(U_1)=U_1$ for any $a\in U(n+1)$.
From Proposition \ref{so} we get that for some integer number $l\in [0,n+1],$ $U$ has $l$ eigenvalues, equal
to ${\bf i},$ and $n+1-l$ eigenvalues, equal to $-{\bf i}.$ Therefore, $\lambda_i=1-\alpha$ for $i\leq l$ and
$\lambda_i=-1-\alpha$ for $i \geq l+1$. Since
$\sum_{i=1}^{n+1}\lambda_i =0$, we get $\alpha =2l/(n+1)-1$.
Let $V_l$ be the matrix $\widetilde{U}$ for a given $l$, $l=0,\dots,n+1$.
Therefore, $U$ is a sum of the matrix $\frac{2l-n-1}{n+1}\diag
({\bf i}, {\bf i},\dots,{\bf i})$ and a matrix from
$\Ad(U(n+1))$-orbit of the matrix $V_l$ (for $l=0$ and $l=n+1$ the matrix $V_l$ is zero).
This orbit is represented
as the homogeneous space $SU(n+1)/S(U(l)\times U(n+1-l))$. For
different $l$ such orbits are disjoint (see Example 8.111 in
\cite{Bes}), which proves the proposition.
\end{proof}

\begin{remark}
Note that for any given vector $v\in S^{2n+1}_{x_0}$ it is always possible to choose in the proof of Theorem
\ref{osn_sp} a unit Killing vector field from any principal orbit
(i.~e. of maximal dimension) among orbits in (\ref{orb}), projecting to $v$.
If $n+1=2(k+1)$, then the unique principal orbit is exactly
the complex Grassmannian from Proposition \ref{uhlensu}.
\end{remark}

\begin{remark}
\label{spin7}
Applying an argument, similar to the one in Sections \ref{S} and \ref{Grassmann}, and the fact that the Clifford
algebra $Cl_6$ is isomorphic to the algebra $\mathbb{R}(8)$ of real $(8\times 8)$-matrices \cite{Hus}, one can prove
that normal homogeneous space $(S^7=Spin(7)/G_2, g_{\ccan})$ in Tables 1 and 2 is $Spin(7)$-Clifford-Wolf homogeneous
and $UKVF(\mathfrak{spin}(7),7)$ is homeomorphic to the real Grassmannian $G_+(7,2)$.
\end{remark}

\newpage

\section{The spheres $(S^{15}=Spin(9)/Spin(7),\psi_t)$}
\label{S}

We study here the most difficult, but at the same time the most interesting case, which involves essentially,
besides other tools, the Clifford algebras $Cl^n$ and the Cayley algebra $\mathbb{C}a$ of octonions. At first
we shall discuss briefly only notions and properties of very general nature, which we really need here.

Let $(\mathbb{R}^n, (\cdot, \cdot))$ be $n$-dimensional Euclidean space with standard inner (scalar) product
$(\cdot, \cdot)$ and orthonormal basis $\{e_1,\dots, e_n\}$. Then the Clifford algebra $Cl^n$ for
\linebreak
$(\mathbb{R}^n, (\cdot, \cdot))$ $\bigl($more exactly, for $(\mathbb{R}^n, -(\cdot, \cdot))$$\bigr)$ is an (unique)
associative
algebra over field $\mathbb{R}$, containing as a subalgebra the field $\mathbb{R}$, with product operation $\cdot\,$,
an extension of a bilinear product $\cdot$ over
$\mathbb{R}^n\times \mathbb{R}^n$ with relation
\begin{equation}
\label{proper}
x\cdot y + y\cdot x=-2(x,y)1,\quad 1\in \mathbb{R},
\end{equation}
such that any possible relation in $(Cl^n,\cdot)$ is a corollary of the relation (\ref{proper}). Note that
\begin{equation}
\label{ex}
Cl^1 \cong \mathbb{C},\quad Cl^2 \cong \mathbb{H},\quad Cl^8 \cong \mathbb{R}(16),
\end{equation}
where $\mathbb{R}(16)$ is the algebra of real $(16\times 16)$-matrices \cite{Hus}. It is clear from definition that
for any $m,$ where $1\leq m\leq n,$ $Cl^m$ is a subalgebra of $Cl^n.$ The algebra $Cl^n$ admits
$\mathbb{Z}_2$-grading $Cl^n=Cl^n_{\overline{0}}\oplus Cl^n_{\overline{1}},$ where $Cl^n_{\overline{0}}$ is its
subalgebra, generated by elements $x\cdot y,$ where $x,y \in \mathbb{R}^n$. For $n>1$ there exists a unique isomorphism
$I_n: Cl^{n-1} \cong Cl^n_{\overline{0}}$ such that $I_n(x)=x\cdot e_n$ if $x\in \mathbb{R}^{n-1}$ and $I_n$ coincides
on $Cl^{n-1}_{\overline{0}}$ with composition of natural inclusions $Cl^{n-1}_{\overline{0}}\subset Cl^{n-1}$ and
$Cl^{n-1}\subset Cl^n.$

\begin{prop}
\label{bivector}
Let $v,w$ and $\{f_1,\dots, f_n\}$ be respectively two orthogonal vectors and an orthonormal basis in
$(\mathbb{R}^n,(\cdot,\cdot)).$ Then
\begin{equation}
\label{minus}
v\cdot w=-w\cdot v,
\end{equation}
the product $v\cdot w$ is uniquely represented in the form
\begin{equation}
\label{compon}
v\cdot w=\sum_{1\leq i< j\leq n}\gamma_{ij}(f_i\cdot f_j),\quad \alpha_{ij}\in \mathbb{R},
\end{equation}
and the components $\gamma_{ij}$ are calculated by the same rule as the components of the bivector
$v \wedge w$.
\end{prop}

\begin{proof}
Formula (\ref{minus}) follows from the relation (\ref{proper}). Let us suppose that
$$v=\sum_{i=1}^n \alpha_if_i,\quad w=\sum_{j=1}^n \beta_jf_j.$$
Then, using the rules (\ref{minus}) and (\ref{proper}), we get
$$v\cdot w = \left(\sum_{i=1}^n \alpha_if_i\right)\cdot \left(\sum_{j=1}^n \beta_jf_j\right)=$$
$$\sum_{1\leq i< j\leq n}(\alpha_i\beta_j(f_i\cdot f_j)+\beta_i\alpha_j(f_j\cdot f_i))+
\sum_{i=1}^n\alpha_i\beta_i(f_i\cdot f_i)=$$
$$\sum_{1\leq i< j\leq n}(\alpha_i\beta_j-\beta_i\alpha_j)(f_i\cdot f_j)-\sum_{i=1}^n \alpha_i\beta_i=$$
$$\sum_{1\leq i< j\leq n}(\gamma_{ij}:=\alpha_i\beta_j-\beta_i\alpha_j)(f_i\cdot f_j).$$
\end{proof}

The following (known) proposition easily follows from Proposition \ref{spso}.

\begin{prop}
\label{spin}
The linear span in $Cl^n$ of elements $x\cdot y,$ where $x,y$ are orthogonal elements in $(\mathbb{R}^n,(\cdot,\cdot))$,
is a Lie algebra with respect to operation
$$[W,V]=W\cdot V-V\cdot W,$$
isomorphic to the Lie algebra $\mathfrak{spin}(n)\cong \mathfrak{so}(n).$
\end{prop}

In force of Propositions \ref{spin} and \ref{bivector}, we shall denote the linear span, mentioned in
Proposition \ref{spin},
as $\mathfrak{spin}(n)$ and sometimes call its elements \textit{bivectors}.

\begin{definition}
An element $W\in \mathfrak{spin}(n)$ is called \textit{simple} if it can be represented in the form $W=v\cdot w$,
where $v,w$
are orthogonal vectors in $(\mathbb{R}^n,(\cdot,\cdot))$.
\end{definition}

\begin{lemma}
\label{simple}
An element $W\in \mathfrak{spin}(n)$ is simple if and only if $W^2=-C^2\cdot 1$ for some real number $C\geq 0$.
In addition, if $W=v\cdot w,$ then $C$ is equal to the volume of the rectangle constructed on the vectors $v,w$.
\end{lemma}

\begin{proof}
Let us suggest at first that $W=v\cdot w$ for orthogonal vectors $v,w\in (\mathbb{R}^n,(\cdot,\cdot))$. Then, using the
relation (\ref{proper}), we get
$$W^2=(v\cdot w)\cdot(v\cdot w)=-(v\cdot v)\cdot(w\cdot w)=-(-(v,v))(-(w,w))\cdot 1 = - (v,v)(w,w)1.$$
This proves the necessity in proposition and its last statement.

Let $W\in \mathfrak{spin}(n)$ be not simple. Then by a statement from paper \cite{Koz}, formulated there
for bivectors,
there exists an orthonormal basis $f_1, \dots, f_n$ in $(\mathbb{R}^n,(\cdot,\cdot))$ such that
$$
W=\sum_{l=1}^{m} a_l (f_{2l-1}\cdot  f_{2l}),\quad\mbox{where}\quad a_l\neq 0\quad\mbox{and}\quad m>1.
$$
Then one can easily check that $W^2\neq -C^2\cdot 1$ for any real $C,$  since $W^2$ will contain a
nonzero ``four-vector'' besides a real number. For example, if $m=2$ then
$$W^2=-(a_1^2+a_2^2)+2a_1a_2 (f_1\cdot f_2\cdot f_3\cdot f_4).$$
\end{proof}

The algebra $Cl^n$ contains in itself not only the Lie algebra $\mathfrak{spin}(n)\subset Cl^n_{\overline{0}}$,
but also the spinor group $(Spin(n),\cdot)$ as a Lie subgroup in the group of all invertible elements
$((Cl^n)^{\times},\cdot)$, see details, for example, in Onishchik's book \cite{On}.

Using special action of the Caley algebra $\mathbb{C}a=(\mathbb{R}^8,(\cdot,\cdot))$ on
$\mathbb{R}^{16}=\mathbb{C}a \oplus \mathbb{C}a,$ one gets
(see, for example, \cite{On}) special isomorphism $\phi: \quad Cl^8 \cong \mathbb{R}(16)$. The
composition of isomor\-phisms $(I_9)^{-1}: Cl^9_{\overline{0}}\cong Cl^{8}$ and $\phi$ naturally
induces an exact representation $\theta: \mathfrak{spin}(9)\rightarrow \mathfrak{gl}(16)$ of Lie algebra
$\mathfrak{spin}(9)$.
It is very important that

1) $\theta(\mathfrak{spin}(9))\subset \mathfrak{so}(16);$

2) $\theta$ is a \textit{spinor representation}, i.e. $\theta$ is induced by (unique) exact representation
$\Theta: Spin(9)\rightarrow SO(16)$;

3) $Spin(9):=\Theta(Spin(9))$ acts transitively on $S^{15};$

4) the isotropy subgroup $H$ of $Spin(9)$ at the point $x_0=(1,0,\dots, 0)^T\in S^{15}$ is isomorphic to the Lie
group $Spin(7)$ \cite{On};

5) the Lie algebra $\mathfrak{h}:=\mathfrak{spin}(7)$ of Lie subgroup $H$ is not the standard inclusion of
$\mathfrak{so}(7)$ into
$\mathfrak{spin}(9)=\mathfrak{so}(9)$, but its image $\tau(\mathfrak{so}(7))$ under an outer automorphism $\tau$
of Lie algebra
$\mathfrak{so}(8)$, so-called \textit{the triality automorphism} of order $3$, where
$\mathfrak{so}(7)\subset \mathfrak{so}(8)\subset \mathfrak{so}(9)=\mathfrak{spin}(9)$ are standard inclusions;

6) $\tau$ is induced by a rotation symmetry $\sigma \in S_3$ of Dynkin diagram $D_4$ (which is a tripod) of the
Lie algebra $\mathfrak{so}(8)$;

7) $\tau$ can be defined, using the Cartan's \textit{triality principle} (see e.~g. a very clear presentation
in paper \cite{GlZ} by H.~Gluck and W.~Ziller), which is also based on
$\mathbb{C}a$.

Thus one gets the homogeneous space $S^{15}=Spin(9)/Spin(7)$.

\begin{lemma}
\label{simple1}
An element $U\in \mathfrak{spin}(9)\subset \mathfrak{so}(16)$ is a unit Killing vector field on $S^{15}$ if and only
if it can be represented as a product $U=v\cdot w$ of orthonormal vectors in $(\mathbb{R}^9,(\cdot,\cdot))$.
\end{lemma}

\begin{proof}
By Lemma \ref{simple}, $U=v\cdot w$ for orthonormal vectors in $(\mathbb{R}^n,(\cdot,\cdot))$ if and only if
$U^2=-1$.  We identify $U$ with its image $\theta(U)\in \mathfrak{so}(16)$.
Since $\phi: Cl^8\cong \mathbb{R}(16)$ is isomorphism and $(I_9)^{-1}(\mathfrak{spin}(9))\subset Cl^8$,
then previous equality is equivalent to the equality $U^2=-\Id$, which in turn by Proposition \ref{so}
is equivalent to the statement that $U$ is a unit Killing vector field on $S^{15}$.
\end{proof}

Since we work with the representation of $\mathfrak{spin}(9)$ in $\mathfrak{so}(16)$,
we can consider the ($\Ad(Spin(9))$-invariant) restriction to $\mathfrak{spin}(9)$
of the inner product (\ref{innerprod}) $\langle \cdot, \cdot \rangle$ on
$\mathfrak{so}(16)$. Then we have $\Ad(Spin(7))$-invariant $\langle\cdot, \cdot\rangle$-orthogonal decomposition
\begin{equation}\label{spdec}
 \mathfrak{g}=\mathfrak{spin}(9)=\mathfrak{spin}(8)\oplus \mathfrak{p}_1=
 \mathfrak{spin}(7)\oplus \mathfrak{p}_2\oplus \mathfrak{p}_1=\mathfrak{h}\oplus \mathfrak{p},
\end{equation}
where the modules $\mathfrak{p}_i$ are $\Ad(Spin(7))$-irreducible,
$[\mathfrak{p}_2,\mathfrak{p}_1]\subset \mathfrak{p}_1$ and
$[\mathfrak{p}_2,\mathfrak{p}_2]\subset \mathfrak{spin}(7)$.
Some convenient $\langle\cdot, \cdot\rangle$-ortogonal basis in $\mathfrak{spin}(9)$, compatible with this
decomposition, which we shall discuss shortly and use later, can be found in \cite{Volp1}.

\begin{theorem}\label{CWhom_spin}
The Euclidean sphere $S^{15}$ is Clifford-Wolf homogeneous with respect to
the Lie group $Spin(9)\subset SO(16)$.
\end{theorem}

\begin{proof}
We should prove that for every tangent vector $u\in S_{x_0}^{15}$ there is a Killing vector field
$U \in \mathfrak{spin}(9)$ of constant length on $S^{15}$ such that $U(x_0)=u$. We can identify
$S_{x_0}^{15}$ with $\mathfrak{p}$. Let $p:\mathfrak{g}\rightarrow \mathfrak{p}$ be corresponding linear projection.

For any unit vector $v\in \mathbb{R}^8$ let $v^{\perp}$ be the orthogonal compliment to $v$ in $\mathbb{R}^8$.
Then we get $7$-dimensional linear subspace
$V_v=\{v\cdot u, u\in v^{\perp}\}\subset \mathfrak{spin}(8)\subset \mathfrak{spin}(9)$.
Analogously we define $8$-dimensional subspace $W_w\subset \mathfrak{spin}(9)$
for any unit vector $w$ in $\mathbb{R}^9$. It is clear that $V_v$ and $W_w$ consist of simple bivectors.
Therefore by Lemma~\ref{simple1}, the linear space $V_v$ consists of Killing vector fields in $\mathfrak{so}(16)$
of constant length on $S^{15},$ in other words, $V_v$ is some $7$-dimensional \textit{Clifford-Killing space}
in $\mathfrak{so}(16)$ \cite{BerNik5}. Analogously, $W_w$ is some
$8$-dimensional Clifford-Killing space in $\mathfrak{so}(16)$.

It is clear that
\begin{equation}
\label{isom}
p: V_v\rightarrow \mathfrak{p}_2
\end{equation}
has zero kernel. Otherwise the intersection $\mathfrak{h}\cap V_v$ has non-zero Killing vector field of constant
length on $S^{15},$ which is impossible. Since $\mathfrak{p}_2$ and $V_v$ are both 7-dimensional, we get a linear
isomorphism $p:=p_v$ in (\ref{isom}). So, for any vector $u\in \mathfrak{p}_2$ and any $v\in S^{7}\subset \mathbb{R}^8$
there is a Killing vector field $K(v,u)\in V_v$ such that $p(K(v,u))=u.$ In fact, $\mathfrak{p}_1=W_{e_9}$
(see e.~g. \cite{On, Volp1}), so it is itself a Clifford-Killing space in $\mathfrak{so}(16)$.

Now let $W\in \mathfrak{p}$ be any vector. Then $W=W_1+W_2,$ where $W_i\in \mathfrak{p}_i, i=1,2.$ We can suggest that
$W_1\neq 0$ and $W_2\neq 0.$ Then $W_1=e_9\cdot v=w\cdot (-|v|)e_9\in W_{w},$ where $w=v/|v|$ for some
$v\in \mathbb{R}^8$. By the previous consideration, there is a vector $Z_2=w\cdot u\in V_w\subset W_w$ such that
$p(Z_2)=W_2.$ Then the element $W_1+Z_2\in W_w \subset \mathfrak{spin}(9)$ is a Killing vector field of constant
length on $S^{15}$ such that $p(W_1+Z_2)=W.$
\end{proof}

\begin{remark}
\label{comp}
At first, we proved Theorem \ref{CWhom_spin} with the help of a computer, presenting an explicit expression
for Killing vector fields from $\mathfrak{spin}(9)$ of constant length on $S^{15}$, projecting to any given tangent
vector in $S^{15}_{x_0}$ under additional requirement that may be only one nonzero vector component in $\mathfrak{p}_2$.
This is sufficient, since $Spin(7)$ acts transitively on spheres with zero center in $\mathfrak{p}_2$.  For vectors in
$\mathfrak{p}_1,$ only vectors themselves were obtained. We used an explicit expression for the embedding monomorphism
$\theta: \mathfrak{spin}(9)\rightarrow \mathfrak{so}(16)$, provided by T.~Friedrich in his paper \cite{Fried}. (Another
expression for $\theta$ is given in \cite{J} in terms of so-called \textit{Clifford cross-section}
$\eta: S^{15}\rightarrow V^{16}_9$, see pp. 3 and 4 in \cite{J}.)
The obtained expressions for Killing vector fields turned out to be rather complicated but verifiable by hand with some
difficulties. To get a required Killing vector field $U\in \mathfrak{spin}(9)\subset \mathfrak{so}(16)$ of constant
length on $S^{15}$ one needs to solve the matrix equation $U^2=- s^2\Id$ for desired skew-symmetric $(16\times 16)$-matrix
$U$.
This gives $(14\cdot 15)/2+14=119$ scalar equations. It is also not difficult to trace on a computer the condition
for skew-symmetric matrix $U$ to be a $\delta$-vector at the point $(1,0,\dots, 0)$ of any round Euclidean sphere,
used in the proof of Theorem \ref{delta_su}.
\end{remark}

\begin{remark}
\label{unique}
Note that the Clifford-Killing space $\mathfrak{p}_1$ in $\mathfrak{so}(16)$ has maximal possible dimension $8$;
also $15$ is minimal dimension for Euclidean spheres having Clifford-Killing spaces
of dimension more than $7$ \cite{Hus}.
\end{remark}

\begin{prop}
\label{spso}
Let $\{f_1,\dots, f_n\}$ be any orthonormal basis in $(\mathbb{R}^n,(\cdot,\cdot))$. Then a unique linear map
$L: \mathfrak{spin}(n)\rightarrow \mathfrak{so}(n)$ such that $L(f_i\cdot f_j)=2F_{ji}:=2(E_{ji}-E_{ij})$
{\rm(}where $E_{ji}$ is a $(n\times n)$-matrix, composed of zeros and only one unit at the place $(j,i)${\rm)},
is an isomorphism of
Lie algebras.
\end{prop}

\begin{proof}
It follows from Proposition \ref{bivector} that $f_i\cdot f_j$, where $1\leq i<j\leq n$, constitute a basis in
$\mathfrak{spin}(n);$ $F_{ji}$ for the same indices constitute a basis in $\mathfrak{so}(n)$. Therefore, $L$ is a
linear isomorphism of vector spaces. Since $f_i\cdot f_j=-f_j\cdot f_i$ for $i\neq j$, we don't need
to care further about the order
of indices $i,j$. It is known that if for all indices $i,j,k,l$, where $i\neq j$ and $k\neq l$, there are no
equal indices or there are two pairs of equal indices, then $[F_{ji},F_{lk}]=0$; one can check directly that
in this case
also $[f_i\cdot f_j,f_k\cdot f_l]=0.$ So it is sufficient to consider the case of indices $i,j,i,k.$ Then
$$[f_i\cdot f_j,f_i\cdot f_k]=f_i f_j f_i f_k-f_i f_kf_i f_j=f_jf_k-f_kf_j=2f_j\cdot f_k,$$
$$[2F_{ji},2F_{ki}]=4[(E_{ji}-E_{ij})(E_{ki}-E_{ik})-(E_{ki}-E_{ik})(E_{ji}-E_{ij})]=$$
$$4(-E_{jk}+E_{kj})=4F_{kj}=L(2(f_j\cdot f_k)).$$
These calculations imply the proposition.
\end{proof}

\begin{prop}
\label{scalar}
Let $\{f_1,\dots, f_9\}$ be an orthonormal basis in $(\mathbb{R}^9,(\cdot,\cdot))$ and
$$U=\sum_{1\leq i<j\leq 9}\gamma^U_{ij}(f_i\cdot f_j),\quad V=\sum_{1\leq i<j\leq 9}\gamma^V_{ij}(f_i\cdot f_j)$$
are two elements in $\mathfrak{spin}(9).$ Then
\begin{equation}
\label{sc}
\langle U, V \rangle=8\sum_{1\leq i<j\leq 9}\gamma^U_{ij}\gamma^V_{ij}.
\end{equation}
\end{prop}

\begin{proof}
Since $SO(9)$ is a simple Lie group, then any two of $\Ad(SO(9))$-invariant inner products on $\mathfrak{so}(9)$ are
proportional. Then it follows from known fact that $F_{ji}$ and $F_{lk}$ are $\langle \cdot, \cdot \rangle$-
orthogonal if and only if $\{j,i\}\neq \{l,k\}.$ This, together with Proposition~\ref{spso}, imply that
$f_i\cdot f_j$ and $f_k\cdot f_l$, where $1\leq i< j\leq 9$ and $1\leq k<l\leq 9,$ are not orthogonal if and
only if they coincide. Then we need to check (\ref{sc}) only for the cases when $U=V=f_i\cdot f_j.$ Let us identify
$f_i\cdot f_j$ with $\theta(f_i\cdot f_j).$ Then by Lemma \ref{simple1} and Proposition \ref{so}
$(f_i\cdot f_j)^2=-\Id.$ Therefore by definition (\ref{innerprod}) of $\langle \cdot, \cdot \rangle,$
$$\langle f_i\cdot f_j, f_i\cdot f_j \rangle = - \frac{1}{2}\trace ((f_i\cdot f_j)^2)=
\frac{1}{2}\trace (\Id)=8.$$
Here we have only one nonzero component $\gamma^U_{ij}=\gamma^V_{ij}=1,$ so the equality (\ref{sc}) and proposition
are proved.
\end{proof}

For the following proposition and Lemma \ref{new_spin2} we need to know bases of subspaces $\mathfrak{h}$
and $\mathfrak{p}_2.$ For this goal we use very symmetric orthogonal bases,
constructed by the first author's former student
D.E.~Volper in paper \cite{Volp1} (see pp. 226 and 227 respectively).
We need only add $1$ to all low indices for bivectors there and interchange $\mathfrak{p}_1$ and $\mathfrak{p}_2$. These
bases have interesting property: if one multiplies the vector $X_4\in \mathfrak{p}_2$ by -1,
then, after adding to any vector of the basis with number $i, 1\leq i \leq 7,$ for $\mathfrak{p}_2$ all three vectors
in $i$-th line for the basis in $\mathfrak{h}$, one gets a simple bivector, i.~e. an element in
$\mathfrak{spin}(8)\subset \mathfrak{so}(9),$ which gives a Killing vector field of constant length on $S^{15}$.

\begin{prop}
\label{vp2}
For any vector $X$ in $\mathfrak{p}_2,$
$$(X,X)_t=\frac{t}{2}\langle X, X\rangle,$$
where $(\cdot,\cdot)_t$ is inner product on $\mathfrak{p},$ corresponding to Riemannian metric
$\psi_t$ on $S^{15}$.
\end{prop}

\begin{proof}
Since $\Ad(Spin(7))$ acts irreducibly on $\mathfrak{p}_2,$ it is sufficient to check the statement for any nonzero
vector $X\in \mathfrak{p}_2.$ Let us take the first vector
$$X_1=e_7e_8-e_1e_2-e_3e_4-e_5e_6,$$
of the basis in $\mathfrak{p}_2,$ constructed in \cite{Volp1}, see also \cite{Fried}. Using the equality (\ref{sc}),
we get $\langle X_1, X_1\rangle=32.$ Adding to $X_1$ the sum
$$Y:=(e_1e_2+e_7e_8)+(e_3e_4+e_7e_8)+(e_5e_6+e_7e_8)$$
of three vectors from the basis in $\mathfrak{h}=\mathfrak{spin}(7),$ we get the vector
$$V=X_1+Y=4e_7e_8\in \mathfrak{spin}(8)\subset \mathfrak{spin}(9)$$
of constant length $C$ on $S^{15}.$ By Lemma \ref{simple1}, $C^2=16.$ It is clear that $X_1 x_0=V x_0$
for initial point $x_0\in S^{15}.$ Taking into account that $X_1 x_0$ is tangent to the ($7$-dimensional)
fiber at $x_0$ of the Hopf fibration
$pro: S^{15}\rightarrow S^8$, we get that $(X_1,X_1)_t=16t$. This proves the proposition.
\end{proof}

\begin{corollary}
\label{scpsi}
The inner product $(\cdot,\cdot)_t$ on $\mathfrak{p}$, corresponding to the Riemannian metric $\psi_t$ on
$Spin(9)/Spin(7)$,
is defined by formula
\begin{equation}\label{spingen}
(\cdot, \cdot)_t=\frac{1}{8}\langle \cdot , \cdot \rangle |_{\mathfrak{p}_1}+
\frac{t}{2} \langle \cdot , \cdot \rangle |_{\mathfrak{p}_2}.
\end{equation}
\end{corollary}

\begin{proof}
In view of Proposition \ref{vp2}, we need to check equality (\ref{spingen}) only for any nonzero vector
$W\in \mathfrak{p}_1.$ Let us take $W=e_8e_9.$ Then by formula (\ref{sc}), $\langle W,W\rangle=8.$ On the
other hand, in view of Lemma \ref{simple1}, $W$ defines unit Killing vector field $W$ on $S^{15},$ and $Wx_0$
is orthogonal to the fiber of the Hopf fibration $pro$ at the point $x_0$. Therefore, $(W,W)_t=1$. This
implies equality (\ref{spingen}) for $W$.
\end{proof}

\begin{remark}
It follows from Propositions \ref{spso} and \ref{scalar} that one needs to multiply the coefficients in (\ref{spingen})
by 2 when using the inner product (\ref{innerprod}) for the Lie algebra $\mathfrak{spin}(9)\cong \mathfrak{so}(9)$
itself.
\end{remark}

\smallskip

Any $Spin(9)$-invariant metric on $S^{15}=Spin(9)/Spin(7)$ is generated (up to homothety) by the inner product on
$\mathfrak{p}$ of the form (\ref{spingen}) for some $t>0$. Note that for $t=1$ we get a metric of constant curvature $1$
on $S^{15}$ and for $t=1/4$ we get a $Spin(9)$-normal homogeneous metric on $S^{15}=Spin(9)/Spin(7)$.
Now, we obtain directly from Corollary \ref{scpsi}, Theorems \ref{body}, \ref{CWhom_spin} and Corollary \ref{cor2}

\begin{prop}\label{new_spin1}
The homogeneous Riemannian space $(S^{15}=Spin(9)/Spin(7),\psi_t)$ is
$Spin(9)$-generalized normal homogeneous for all $t \in [1/4,1]$.
\end{prop}

\begin{lemma}\label{new_spin2}
The metric $\psi_t$ with $t>1$ is not $Spin(9)$-generalized normal homogeneous.
\end{lemma}

\begin{proof}
Suppose that the metric $\psi_t$ is $Spin(9)$-generalized normal homogeneous
on $S^{15}=Spin(9)/Spin(7)$.
By Proposition 22 in \cite{BerNik} we know that
for every $X \in
\mathfrak{p}_1$, $Y \in \mathfrak{p}_2$ the inequality
$$
x_1 \langle [[Y,X],X]_\mathfrak{h},[[Y,X],X]_\mathfrak{h} \rangle
\geq (x_2-x_1) \langle
[[Y,X],X]_{\mathfrak{p}_2},[[Y,X],X]_{\mathfrak{p}_2} \rangle.
$$
holds, where $x_1=\frac{1}{8}$ and $x_2=\frac{t}{2}$ by Corollary \ref{scpsi}.

Now we consider $X=e_2\cdot e_9 \in \mathfrak{p}_1$ and
$Y=e_1\cdot e_2+e_3\cdot e_4+e_5\cdot e_6-e_7\cdot e_8 \in \mathfrak{p}_2$ (this is the vector $-X_1$ on p. 227 in
\cite{Volp1}, see also \cite{Fried}). It is easy to check that $[[Y,X],X]]=-4 e_1\cdot e_2=Z-Y$, where the vector
$Z=-3e_1\cdot e_2+e_3\cdot e_4+e_5\cdot e_6-e_7\cdot e_8=
-3(e_1\cdot e_2+e_7\cdot e_8)+(e_3\cdot e_4+e_7\cdot e_8)+(e_5\cdot e_6+e_7\cdot e_8)\in \mathfrak{spin}(7)$
(all vectors appeared here in brackets are basis vectors for $\mathfrak{h}$ in the first line of basis vectors for
$\mathfrak{h}$ on p. 226 in \cite{Volp1}, see also \cite{Fried}).
Since $\langle -Y, -Y \rangle=32$ and $\langle Z, Z \rangle=96$, then we get
$\frac{1}{8}\cdot 96 \geq \frac{4t-1}{8}\cdot 32$, i.~e. $t\leq 1$.
\end{proof}

\bigskip

\begin{lemma}\label{new_spin3}
The metric $\psi_t$ with $t<1/4$ is not $Spin(9)$-generalized normal homogeneous.
\end{lemma}

\begin{proof}
The metric $\psi_t$ is generated by the inner product (\ref{spingen}).
Since the pair $(\mathfrak{spin}(7)\oplus \mathfrak{p}_2,\mathfrak{spin}(7))$ is the symmetric pair
$(\mathfrak{so}(8),\mathfrak{so}(7))$
(that corresponds to a two-point homogeneous space $S^7=SO(8)/SO(7)$),
then for any non-trivial $X \in \mathfrak{p}_2$ we see that $Z_{\mathfrak{h}}(X)$ (see Lemma \ref{deltal1})
is isomorphic to $\mathfrak{so}(6)$.
By Corollary \ref{deltal2} we get that $X$ is $\delta$-vector and
$(X,[U,[U,X]]_{\mathfrak{p}})+([U,X]_{\mathfrak{p}},[U,X]_{\mathfrak{p}})\leq 0$
for all $U\in \mathfrak{spin}(9)$. Take any $U\in \mathfrak{p}_1$ such that $[U,X]\neq 0$,
then $[U,X] \in \mathfrak{p}_1$ and
\begin{eqnarray*}
0\geq (X,[U,[U,X]]_{\mathfrak{p}})+([U,X]_{\mathfrak{p}},[U,X]_{\mathfrak{p}})=\\
\frac{t}{2}\langle X,[U,[U,X]]\rangle+\frac{1}{8}\langle [U,X]_{\mathfrak{p}},[U,X]_{\mathfrak{p}}\rangle=\\
-\frac{t}{2}\langle [U,X],[U,X] \rangle+\frac{1}{8}\langle [U,X],[U,X]\rangle=\frac{1-4t}{8}\langle [U,X],[U,X] \rangle.
\end{eqnarray*}
Since $[U,X]\neq 0$, we get $t \geq 1/4$.
\end{proof}

From Proposition \ref{new_spin1}, Lemma \ref{new_spin2}, and Lemma \ref{new_spin3} we obviously get

\begin{theorem}\label{new_spin}
The homogeneous Riemannian space $(S^{15}=Spin(9)/Spin(7),\psi_t)$ is
$Spin(9)$-generalized normal homogeneous if and only if $t \in [1/4,1]$.
\end{theorem}

\section{Spaces of unit Killing vector fields on $S^{15},$ connected with $\mathfrak{spin}(9)$}
\label{Grassmann}

\begin{prop}
\label{int3}
The Grassmannian $G_+(9,2)=SO(9)/(SO(2)\times SO(7))$ of oriented real 2-planes in $\mathbb{R}^9$ can be
interpreted as the space $UKVF(\mathfrak{spin}(9),15)$ of all Killing vector fields of unit length on $S^{15}$,
lying in the Lie algebra $\mathfrak{spin}(9)\subset \mathfrak{so}(16)$.
\end{prop}

\begin{proof}
It follows from Lemmas \ref{simple} and \ref{simple1} that an element $U\in \mathfrak{spin}(9)\subset \mathfrak{so}(16)$
gives a unit Killing vector field on $S^{15}$ if and only if it can be presented in the form $U=v\cdot w,$ where $v,w$
are orthonormal vectors in $(\mathbb{R}^9,(\cdot,\cdot))$.

Now one can check easily, directly or using Proposition \ref{bivector}, that two products
\linebreak
$v_1\cdot w_1$ and
$v_2\cdot w_2$ for pairs of orthonormal vectors in $(\mathbb{R}^9,(\cdot,\cdot))$ coincide if and only if these pairs
define one and the same oriented 2-plane in $\mathbb{R}^9.$ This finishes the proof of proposition.
\end{proof}

Now we want to describe the space of unit Killing vector fields on $S^{15},$ lying in the Lie
algebra $\mathfrak{spin}(9),$ whose image under above projection $p: \mathfrak{spin}(9)\rightarrow \mathfrak{p}$
is situated in $\mathfrak{p}_2$ (respectively, $\mathfrak{p}_1$).

\begin{prop}
\label{int4}
The Grassmannian $G_+(8,2)=SO(8)/(SO(2)\times SO(6))$ of oriented real 2-planes in $\mathbb{R}^8$ can be interpreted
as the space of all Killing vector fields of unit length on $S^{15}$, lying in the Lie subalgebra
$\mathfrak{spin}(8)\subset \mathfrak{spin}(9)\subset \mathfrak{so}(16).$ This also can be considered as the space
of all unit Killing vector fields on $S^{15},$ lying in $\mathfrak{spin}(9)$ and projecting under $p$ into
$\mathfrak{p}_2.$
\end{prop}

\begin{proof}
The first statement is proved in the same way as Proposition \ref{int3}. The second statement follows from the first
statement and relations $p(\mathfrak{spin}(8))=\mathfrak{p}_2,$ $p^{-1}(\mathfrak{p}_2)\subset \mathfrak{spin}(8)$.
\end{proof}

The following proposition gives a scheme for the search of \textit{all} (unit) Killing vector fields on $S^{15},$ lying
in $\mathfrak{spin}(9)$ and projecting into $\mathfrak{p}_2$ (which actually always lie in $\mathfrak{spin}(8)$
by Proposition \ref{int4}), applied in the proof of Theorem \ref{CWhom_spin}.

\begin{prop}
\label{Kil}
Let $S^7$ and $S^6$ be unit spheres respectively in $(\mathbb{R}^8,(\cdot,\cdot))$ and
\linebreak
$(\mathfrak{p}_2,\frac{1}{2}\langle \cdot, \cdot \rangle|_{\mathfrak{p}_2})$.
Then there is the following sequence of real-analytic maps
\begin{equation}
\label{seq}
S^7\times S^6\stackrel{(\Id\circ p_1)\times K}{\longrightarrow}
V_2^8\stackrel{q}{\longrightarrow} G_+(8,2)\stackrel{\incl}{\longrightarrow}
\mathfrak{spin}(8)\stackrel{p}{\longrightarrow} S^6.
\end{equation}
Here $V_2^8=SO(8)/SO(6)$ is homogeneous Stiefel manifold, consisting of all orthonormal 2-frames in
$\mathbb{R}^8$ \cite{Hus}, $q$ is the canonical projection, $\incl$ is a natural inclusion map given by
Proposition \ref{int4}, $K$ is the map from the proof of Theorem \ref{CWhom_spin}. The first map
$(\Id\circ p_1)\times K$ associates the pair $(v,K(v,u))\in V_2^8$ to a pair $(v,u)$ of unit vectors in
$S^7\times S^6$ ; it is a diffeomorphism. Moreover, for any point
$(v,u)\in S^7\times S^6,$ $\incl(q(((\Id\circ p_1)\times K)(v,u)))$ there is some unit Killing vector field on
$S^{15}$, lying in $\mathfrak{spin}(8)$, which $p$ projects to $u$. Any unit Killing vector vector field
on $S^{15},$ lying in $\mathfrak{spin}(9)$ and projecting to $u\in S^6,$ has this form.
\end{prop}

\begin{proof}
It is clear that there exists the inverse map $f$ to $(\Id\circ p_1)\times K$, which is defined by formula
$f(v,w)=(v,p(\incl(q(v,w)))).$ Obviously, this map is real-analytic. It is enough to prove that $f$ is a diffeomorphism.
At first we define another diffeomorphism $F: V_2^8\rightarrow S^7\times S^6.$ As a corollary of classical results of
Hurwitz-Radon, there is a 7-dimensional Clifford-Killing space $CK_7$ on $S^7$ \cite{BerNik5} with some orthonormal basis
$\{Y_1,\dots , Y_7\}$. For any pair $(v,w)\in V_2^8,$ $w$ is a tangent vector to $S^7$ at the point $v\in S^7.$ So
it can be presented in the form $w=s_1Y_1(v)+\dots + s_7Y_7(v),$ where $s_1^2+\dots +s_7^2=1,$ thus we can identify
$S=(s_1,\dots,s_7)$ with a point in $S^6.$ By definition, $F(v,w)=(v,S)\in S^7\times S^6.$ It is clear that $F$ is
a diffeomorphism.

Now we see that the first component $p_1(f(F^{-1}(v,S)))=v$ of $f\circ F^{-1}$ is identical by $v$ for fixed $S$,
while its second component $p_2(f(F^{-1}(v,S)))=p_2(f(v,w))$ isometrically depends on $w\in v^\perp$, hence on $S$,
under fixed $v,$ because $v\cdot w$ remains in fixed Clifford-Killing space $V_v\subset \mathfrak{spin}(8)$,
while the map (\ref{isom}) is nondegenerate linear. Therefore the differential $D(f\circ F^{-1})$ of $f\circ F^{-1}$
is nondegenerate at any point $(v,S)\in S^7\times S^9.$ Therefore the differential $Df$ of $f$ is nondegenerate at
any point
$(v,w)\in V_2^8.$ By the inverse function theorem, the maps $f$ and $(\Id\circ p_1)\times K$ are
mutually inverse real-analytic diffeomorphisms.

The last statement follows from Proposition \ref{int4}. The statement before it follows from the fact that $f$ and
$(\Id\circ p_1)\times K$ are mutually inverse maps.
\end{proof}

\smallskip

We have an analogous proposition, presenting the scheme for the search of (unit) Killing vector fields on $S^{15}$, lying in
$\mathfrak{spin}(9)$ and projecting into $\mathfrak{p}-\{\mathfrak{p}_1\cup \mathfrak{p}_2\}$, similarly to the second
part in the proof of Theorem \ref{CWhom_spin}. Let $S^{14}$ and $S^7_1$ be respectively unit sphere in
$(\mathfrak{p},(\cdot, \cdot)_1)$
and $(\mathfrak{p}_1,\frac{1}{8}\langle \cdot, \cdot \rangle|_{\mathfrak{p}_1}),$ $S^6$ is the same as in
Proposition \ref{Kil}. Then it is known that $S^{14}$ is the \textit{join}
$S^7_1\ast S^6$ \cite{FR}. This means that $S^{14}$ is the image of continuous map
$J: S^7_1\times S^6\times [0,1]\rightarrow S^{14},$ where $J(x,y,s)=sx + \sqrt{1-s^2}y$. In addition,
$J$ is real-analytic
homeomorphism of $S^7_1\times S^6\times (0,1)$ onto $S^{14}-\{S^7_1\cup S^6\}$.

\begin{prop}
\label{Kil1}
Let $S^7$ be the same as in Proposition \ref{Kil}. Then there is the following sequence of real-analytic maps
$$S^{14}-\{S^7_1\cup S^6\}\stackrel{J^{-1}}{\longrightarrow}
S^7_1\times S^6\times (0,1) \stackrel{g\times \Id\times \Id}{\longrightarrow} S^7\times S^6\times (0,1)
\stackrel{I\times Id}{\longrightarrow}$$
$$ V_2^8\times (0,1) \stackrel{\Id\times J}{\longrightarrow} V_2^9\stackrel{Q}{\longrightarrow}
G_+(9,2)\stackrel{\incl}{\longrightarrow} \mathfrak{spin}(9)\stackrel{p}{\longrightarrow} S^{14}-\{S^7_1\cup S^6\}. $$
Here $g$ maps an element $v\cdot e_9\in S^7_1$ to $v\in S^7,$ $I=(\Id\circ p_1)\times X,$
$(\Id\times J)(v,w,s):=(v,J(e_9,w,s)),$ $Q$ is the canonical projection, $\incl$ is a natural inclusion map given by
Proposition \ref{int3}. Moreover, the composition $f$ of all maps but the last one (in the above diagram), applied
to any vector $u\in S^{14}-\{S^7_1\cup S^6\}$, gives some unit Killing vector field on $S^{15},$ lying in
$\mathfrak{spin}(9),$ which is projected under $p$ to the vector $u$.
\end{prop}

\begin{proof}
This proposition easily follows from Proposition \ref{Kil}.
\end{proof}

\begin{prop}
\label{closure}
The image $A$ of the set $S^{14}-\{S^7_1\cup S^6\}$ in $G_+(9,2)$ under the map $f$ from Proposition \ref{Kil1}
is open and connected in $G_+(9,2)$. Its closure is equal to $G_+(9,2)$ and its boundary consists of two disjoint
connected components, $G_+(8,2)$ and $\mathfrak{p}_1$. A~nonzero vector $u\in \mathfrak{p}$ is a projection under $p$
of unique Killing vector field of constant length on $S^{15}$, lying in $\mathfrak{spin}(9)$, if and only if
$u\notin \mathfrak{p}_2$.
\end{prop}

\begin{proof}
Dimensions of $S^{14}$ and $G_+(9,2)=SO(9)/(SO(2)\times SO(7))$ are both equal to $14$.
The composition of all maps in the diagram from Proposition \ref{Kil1} is identical on the open subset
$S^{14}-\{S^7_1\cup S^6\}\subset S^{14}$; all maps in Proposition \ref{Kil1} are real-analytic.
Then the set $A$ is open in $G_+(9,2)$. The boundary of the set $A$ in $G_+(9,2)$ consists of two closed connected
components. The first, which we denote by $C$, is $\mathfrak{p}_1$, of dimension $8$. Another one, which we denote by
$D$,
consists of unit Killing field on $S^{15}$, projecting to $S^{6}\subset \mathfrak{p}_2$. Therefore, by
Propositions \ref{Kil} and \ref{int4}, $D\subset G_+(8,2)$. Thus the topological dimension of $D$ is no more than
dimension of $G_+(8,2)=SO(8)/(SO(2)\times SO(6)),$ which is equal to 12. So $C\cup D$ cannot divide $G_+(9,2)$.
Let us suppose that the closure of $A$ in $G_+(9,2)$ is not equal to $G_+(9,2)$. Then there is
a point $X,$ lying in nonempty open subset $G_+(9,2)-(A\cup C\cup D).$ Since $C\cup D$ does not divide $G_+(9,2)$,
and $G_+(9,2)$ is arcwise connected, then there is an arc in open subset $G_+(9,2)-(C\cup D)$, joining
point $X$ with arbitrary given point $Y\in A.$ But this is impossible, because $A$ is connected and open in
$G_+(9,2)$,
while $A\cup C\cup D$ is closed in $G_+(9,2)$. Therefore, in addition $D=G_+(8,2)$ also. The last statement follows from
previous ones.
\end{proof}

\begin{remark}
Since $\pi_2(SO(9)/(SO(2)\times SO(7)))=\pi_1(SO(2))=\mathbb{Z}$, the Grassmannian $G_+(9,2)$ is not homeomorphic to
$S^{14},$ see book \cite{FR} by Rokhlin and Fuks.
\end{remark}

\begin{corollary}
\label{intp1}
The space of all unit Killing vector fields on $S^{15}$, lying in the Lie
algebra $\mathfrak{spin}(9)\subset \mathfrak{so}(16)$ and projecting under $p$ into $\mathfrak{p}_1$, is the space
$\mathfrak{p}_1$ itself. In addition, $p$ is identical on $\mathfrak{p}_1$.
\end{corollary}

\begin{corollary}
\label{all}
Proposition \ref{Kil1} gives all unit Killing vector fields on $S^{15}$ from $\mathfrak{spin}(9),$ projecting to
$\mathfrak{p}-(\mathfrak{p}_1\cup \mathfrak{p}_2).$ Thus Propositions \ref{Kil}, \ref{Kil1}, and Corollary \ref{intp1}
altogether present the way to get all
unit Killing vector fields on $S^{15}$ from $\mathfrak{spin}(9),$ projecting to any given point in
$S^{14}\subset \mathfrak{p}.$
\end{corollary}

It follows from Propositions \ref{int1}, \ref{uhlensp}, \ref{uhlensu}, and \ref{int3} that symmetric spaces
$O(2n)/U(n),$ $Sp(n+1)/U(n+1),$ $SU(2(n+1))/S(U(n+1)\times U(n+1))$, and $G_+(9,2)$ can be considered as real-analytic
closed submanifold in $\mathfrak{so}(2n)$, $\mathfrak{sp}(n+1),$ $\mathfrak{su}(2(n+1))$, and
$\mathfrak{spin}(9)\subset \mathfrak{so}(16),$
respectively. Clearly, they do not intersect corresponding isotropy Lie subalgebras $\mathfrak{so}(2n+1)$,
$\mathfrak{sp}(n),$ $\mathfrak{s}(\mathfrak{u}(n+1)\oplus \mathfrak{u}(n+1))$, and $\mathfrak{spin}(7)$.
Similar statements are true for any
separate orbit in (\ref{orb}). Moreover, by Theorems \ref{osn_u}, \ref{osn_sp},  and \ref{CWhom_spin}, the natural
linear projections map respectively  these symmetric spaces, or their union (\ref{orb}) onto unit spheres $S^{2(n-1)}$,
$S^{4n+2},$ $S^{4n+2},$  $S^{14}$, and $S^{2n}$ in tangent spaces to $S^{2n-1},$ $S^{4n+3},$ $S^{4n+3,}$ $S^{15},$ and
$S^{2n+1}$ at their initial points.

Really, it is possible to reconstruct the search for all required unit Killing vector fields on $S^{15}$ from $G_+(9,2),$
and the space $G_+(9,2)$ itself, knowing only the map $f$ from Proposition \ref{Kil1}. There is unique continuous
extension of $f$ to $S^{14}-S^{6},$ which we denote by $f,$ but there is no such extension to all $S^{14}$.
Consider continuous map
$$\phi:= f\circ J: S^7_1\times S^6\times (0,1]\rightarrow G_+(9,2).$$
It has unique continuous extension to all $S^7_1\times S^6\times [0,1],$ which we denote as $\phi$. Now the
algorithm for the search for all required unit Killing vector fields on $S^{15}$ is expressed
by equalities
$$\phi (S^7_1\times S^6\times [0,1])=G_+(9,2);\quad \phi(S^7_1\times S^6\times \{0\})=G_+(8,2);\quad p\circ \phi=J.$$
Since $\phi$ is surjective, continuous, and closed map, we also get the space $G_+(9,2)$ as the quotient space
of $S^7_1\times S^6\times [0,1]$  with respect to $\phi.$

Let us take $s=\frac{1}{2}$ and $O:=J(S^7_1\times S^6\times [0,s)),$ which is open tubular $(\pi/4)$-neighborhood of
$S^6$ in $S^{14}.$ Then the formula
$$\omega (J(v,w,t))=\phi(v,w,2t-1),\quad \frac{1}{2}\leq t\leq 1,$$
correctly defines surjective, continuous, and closed map of the complement $CO$ for the neighborhood $O$ of $S^6$
in $S^{14}$ onto $G_+(9,2).$ So $G_+(9,2)$ is the quotient space of $CO$, and $\omega$ is homeomorphism on $CO-B,$
where $B$ is joint boundary of $O$ and $CO.$

Change $\omega|_B$ by surjective continuous mapping $q\circ F^{-1}\circ h: B\rightarrow G_+(8,2),$
where $h$ is the canonical homeomorphism of $B$ onto $S^7_1\times S^6,$ and $q$ and $F$ are taken from
Proposition \ref{Kil} and its proof. Then the space $G_+(9,2)$ is a result of gluing  $CO$ with $G_+(8,2)$ by the map
$q\circ F^{-1}\circ h.$ One can prove that there exists some real-analytic map $c: G_+(8,2)\rightarrow S^6$ such
that
\begin{equation}
\label{cqf}
c\circ q \circ F^{-1}=p_2,\quad\mbox{where}\quad p_2: S^7\times S^6 \rightarrow S^6
\end{equation}
is the projection to the second factor. In addition, if $r:S^6\rightarrow \mathbb{R}P^6$ is the canonical
projection,
then $r\circ c: G_+(8,2)\rightarrow \mathbb{R}P^6$ is a real-analytic fibration with the fiber $\mathbb{C}P^3.$
Unlike (\ref{cqf}) now we only have more complicated (although analogous) formula
$$(p\circ \omega)(J(v,w,s))= A(v)\circ p_2,\quad\mbox{where}\quad A(v)\in SO(7),$$
for the restriction $p: G_+(8,2)\rightarrow S^6$ of the above linear projection
$p: \mathfrak{spin}(8)\rightarrow \mathfrak{p}_2.$ We don't know whether real-analytic map
$r\circ p: G_+(8,2)\rightarrow \mathbb{R}P^6$ is a fibration
with the fiber $\mathbb{C}P^3.$

\section*{The conclusion}

Despite many dispersed remarks in the main body of the paper, it is appropriate to give some additional remarks
concerning all investigated spaces.

By Proposition \ref{addsymm}, any $Sp(n+1)$-generalized normal homogeneous Riemannian metric on
$S^{4n+3}=Sp(n+1)/Sp(n)$ is proportional to some
metric $\mu_t$ (see paper~\cite{B1} for details).
Therefore, metrics from Table 2 exhaust all generalized normal homogeneous Riemannian metrics on spheres.
On the other hand, all metrics from Table 2 induce generalized normal homogeneous metrics on corresponding
real projective spaces. Metrics obtained in such a way, metrics from
Corollary \ref{cp},  together with normal homogeneous metrics on the projective spaces
$\mathbb{C}P^n=SU(n+1)/S(U(n)\times U(1))$,
$\mathbb{H}P^n=Sp(n+1)/Sp(n)\times Sp(1)$,
and $\mathbb{C}aP^2=F_4/Spin(9)$
exhaust all generalized normal homogeneous Riemannian metrics on projective spaces
(see details in \cite{Zil82} and~\cite{BNN}).

All compact generalized homogeneous spaces studied here have positive sectional curvature.
Omitting details, we refer to papers \cite{Volp}, \cite{Volp1}, and \cite{Volp2} by D.E.~Volper,
where he calculated the exact upper and lower bounds of sectional curvatures for all suitable one-parameter families.
In all these cases, the bounds are some functions of the parameter $t$ and don't depend
on dimension. It is very interesting that these functions coincide for spaces $(S^{4n+3},\mu_t)$ and $(S^{15},\psi_t)$.
There are some related results for spheres in paper \cite{VZ} by L.~Verdiani and W.~Ziller and in some other papers.
As far as we know, the corresponding bounds for family $(S^{4n+3},\mu_{t,s})$ have not been calculated in the literature.
Nevertheless, with the help of some criteria for positivity of sectional curvature from paper \cite{VZ},
it is proved in \cite{B2} that all generalized normal homogeneous spaces from this family also have positive sectional
curvatures.

It should be noted that generalized normal homogeneous Riemannian metrics show a great diversity of properties.
For example, for the families in Theorem \ref{nonnormal}, $\mu_t$ and $\xi_t$ are simultaneously weakly symmetric
and naturally reductive, $\nu_t$ and $\psi_t$ are weakly symmetric but not naturally reductive \cite{Zil82,Zil96}.
The metrics $\mu_{t,s}$ are weakly symmetric (see e.~g. 12.9.2 in \cite{Wolf2007}),
but not naturally reductive \cite{Zil82}.
Yu.G.~Nikonorov proved that geodesic orbit Riemannian metrics on $S^{4n+3}$ with respect to $Sp(n+1)$
are precisely the multiples of metrics $\mu_t$ \cite{NikGOsp}.

\vspace{10mm}

\end{document}